\pdfoutput=1
\documentclass[11pt]{article}

\usepackage[margin=1in]{geometry}
\usepackage[T1]{fontenc}
\usepackage[utf8]{inputenc}

\usepackage{booktabs}
\usepackage{subfigure}
\usepackage{subcaption}
\usepackage[running, mathlines]{lineno}
\usepackage{braket,amsfonts,amssymb}
\usepackage{enumitem,mathtools}
\usepackage{xcolor}
\usepackage{colortbl}
\usepackage{amsmath}
\usepackage{amsthm}
\usepackage{array}
\usepackage{wrapfig}
\usepackage{pgfplots}
\usepackage{algorithmic}
\usepackage{tikz}
\usepackage{graphicx,epstopdf}
\usepackage{mathrsfs}
\usepackage{amsopn}
\usepackage{xspace}
\usepackage{bold-extra}
\usepackage[most]{tcolorbox}
\usepackage{caption}
\usepackage{color}
\usepackage{etoolbox}
\usepackage{titlesec}
\usepackage[colorlinks=true,linkcolor=blue,citecolor=blue,urlcolor=blue]{hyperref}
\usepackage[nameinlink,capitalize]{cleveref}
\usepackage{listings}
\newcommand{\mymin}{\scalebox{0.55}{$\min$}}
\newcommand{\mymax}{\scalebox{0.55}{$\max$}}

\AtBeginDocument{
\addtolength{\abovedisplayskip}{-0.8ex}
\addtolength{\abovedisplayshortskip}{-0.8ex}
\addtolength{\belowdisplayskip}{-0.8ex}
\addtolength{\belowdisplayshortskip}{-0.8ex}
}

\makeatletter
\def\thm@space@setup{\thm@preskip=3pt \thm@postskip=3pt}
\makeatother

\titlespacing*{\section}{0pt}{0.6\baselineskip}{0.2\baselineskip}
\titlespacing*{\subsection}{0pt}{0.6\baselineskip}{0.2\baselineskip}

\numberwithin{equation}{section}
\numberwithin{table}{section}
\numberwithin{figure}{section}

\newcommand{\myin}{\scalebox{0.7}{\text{in}}}
\newcommand{\EQ}{\begin{equation}}
\newcommand{\EN}{\end{equation}}
\newcommand{\EQS}{\begin{equation*}}
\newcommand{\ENS}{\end{equation*}}
\newcommand{\EQA}{\begin{eqnarray}}
\newcommand{\ENA}{\end{eqnarray}}
\newcommand{\EQAS}{\begin{eqnarray*}}
\newcommand{\ENAS}{\end{eqnarray*}}
\newcommand{\AL}{\begin{align}}
\newcommand{\AN}{\end{align}}
\newcommand{\ALS}{\begin{align*}}
\newcommand{\ANS}{\end{align*}}

\newcommand{\mysumsmall}{\sum}
\newcommand{\mysum}{\sum}

\newcommand{\myblue}{\color{black}}

\definecolor{lightblue}{rgb}{0.3, 0.0, 0.9}

\newcommand{\winf}{w_{\scalebox{0.6}{-$\infty$}}}

\newcommand{\Ebb}{\mathbb{E}}
\newcommand{\Nbb}{\mathbb{N}}

\newcommand{\ds}{\displaystyle}
\def\r{\right}
\def\l{\left}

\newcommand{\Rbb}{\mathbb{R}}

\newcommand{\myout}{\scalebox{0.7}{\text{out}}}

\newtheorem{theorem}{Theorem}[section]
\newtheorem{lemma}[theorem]{Lemma}
\newtheorem{proposition}[theorem]{Proposition}
\newtheorem{corollary}[theorem]{Corollary}

\newtheorem{assumption}{Assumption}[section]
\newtheorem{definition}{Definition}[section]

\theoremstyle{remark}
\newtheorem{remark}[theorem]{Remark}

\crefname{hypothesis}{Hypothesis}{Hypotheses}
\Crefname{hypothesis}{Hypothesis}{Hypotheses}
\Crefname{ALC@unique}{Line}{Lines}

\colorlet{texcscolor}{blue!50!black}
\colorlet{texemcolor}{red!70!black}
\colorlet{texpreamble}{red!70!black}
\colorlet{codebackground}{black!25!white!25}

\lstdefinestyle{siamlatex}{%
  style=tcblatex,
  texcsstyle=*\color{texcscolor},
  texcsstyle=[2]\color{texemcolor},
  keywordstyle=[2]\color{texemcolor},
  moretexcs={cref,Cref,maketitle,mathcal,text,email,url},
}

\tcbset{%
  colframe=black!75!white!75,
  coltitle=white,
  colback=codebackground,
  colbacklower=white,
  fonttitle=\bfseries,
  arc=0pt,outer arc=0pt,
  top=1pt,bottom=1pt,left=1mm,right=1mm,middle=1mm,boxsep=1mm,
  leftrule=0.3mm,rightrule=0.3mm,toprule=0.3mm,bottomrule=0.3mm,
  listing options={style=siamlatex}
}

\newtcblisting[use counter=example]{example}[2][]{%
  title={Example~\thetcbcounter: #2},#1}

\newtcbinputlisting[use counter=example]{\examplefile}[3][]{%
  title={Example~\thetcbcounter: #2},listing file={#3},#1}

\DeclareTotalTCBox{\code}{ v O{} }
{
  fontupper=\ttfamily\color{black},
  nobeforeafter,
  tcbox raise base,
  colback=codebackground,colframe=white,
  top=0pt,bottom=0pt,left=0mm,right=0mm,
  leftrule=0pt,rightrule=0pt,toprule=0mm,bottomrule=0mm,
  boxsep=0.5mm,
  #2}{#1}

\patchcmd\newpage{\vfil}{}{}{}
\title{Monotone 2D Integration Scheme for Mean--CVaR Optimization via Fourier-Trained Transition Kernels}

\author{
Duy-Minh Dang\thanks{School of Mathematics and Physics, The University of Queensland, St Lucia, Brisbane 4072, Australia (\href{mailto:duyminh.dang@uq.edu.au}{duyminh.dang@uq.edu.au}).}
\and
Hao Zhou \thanks{School of Mathematics and Physics, The University of Queensland, St Lucia, Brisbane 4072, Australia (\href{mailto:h.zhou3@student.uq.edu.au}{h.zhou3@student.uq.edu.au}).}
}

\date{}

\begin{document}
\maketitle

\begin{abstract}
We present a strictly monotone, provably convergent two-dimensional (2D) integration method for multi-period mean--conditional value-at-risk (mean--CVaR) reward--risk stochastic control in models whose one-step increment law is specified via a closed-form characteristic function (CF). When the transition density is unavailable in closed form, we learn a nonnegative, normalized 2D transition kernel in Fourier space using a simplex-constrained Gaussian-mixture parameterization, and discretize the resulting convolution integrals with composite quadrature rules with nonnegative weights to guarantee monotonicity. The scheme is implemented efficiently using 2D fast Fourier transforms.
Under mild Fourier-tail decay assumptions on the CF, we derive Fourier-domain $L_2$ kernel-approximation and truncation error estimates and translate them into real-space bounds that are used to establish $\ell_\infty$-stability, consistency, and pointwise convergence as the discretization and kernel-approximation parameters vanish. Numerical experiments for a fully coupled 2D jump--diffusion model in a multi-period portfolio optimization setting illustrate robustness and accuracy.
\end{abstract}

\vspace{0.5em}


\section{Introduction}
\label{sec:intro}
Reward--risk optimization arises in many settings where decisions trade off performance against adverse outcomes.
Such formulations appear in environmental and resource management \cite{labadie2004optimal}, supply-chain planning \cite{kleindorfer2005managing}, and more broadly in engineering and applied decision-making \cite{filippi2020conditional}. In finance, the same trade-off is central in multi-period portfolio choice and retirement saving under market and macroeconomic uncertainty \cite{forsyth2020multiperiod}.

A widely used tail-risk measure in such formulations is conditional value-at-risk (CVaR), also known as expected shortfall \cite{rockafellar2000optimization}.
It quantifies tail risk via a conditional expectation beyond a quantile threshold and admits a threshold-based representation that is computationally convenient.
In reward--risk optimization, the reward is typically an expected expected gain or payoff, making it natural to pair this mean term with a CVaR penalty.
The resulting mean--CVaR objective is widely used in multi-period stochastic control \cite{forsyth2020multiperiod,miller2017optimal}.

From a numerical perspective, value functions arising in stochastic control, including those induced by mean--CVaR criteria, are typically nonsmooth.
Since optimal decisions are obtained by comparing numerical value functions across admissible controls, convergence of the discretization is essential.
In particular, if the scheme is not monotone, the computed value function may converge to the wrong solution, leading to unreliable policies
\cite{barles-souganidis:1991, oberman2006convergent}.
This motivates the design of monotone schemes for reward--risk control problems.

Monotone schemes are central to the convergence theory of Barles--Souganidis \cite{barles-souganidis:1991}; however, their construction and analysis are most fully developed in one-dimensional (1D) models, both with and without jumps.
Constructing strictly monotone schemes in multiple dimensions is substantially more challenging, especially under correlation and nonlocal jump effects, and robust treatments of multi-dimensional integro-differential terms remain difficult \cite{MaForsyth2015}.

In many popular stochastic models, the transition density is not available in closed form, while its Fourier transform, i.e.\ the characteristic function (CF) of the
incrementsm is available analytically \cite{Fang2008}.
When CFs are known, Fourier-based methods can then handle complex dynamics efficiently and provide an attractive alternative to finite differences.
Among these, the Fourier-cosine (COS) method is particularly notable for achieving high-order convergence in piecewise smooth problems \cite{Fang2008,ruijter2012two}.
However, for general stochastic control problems, which often involve nonsmooth structures, such convergence is typically unattainable 
\cite{lu2024semi}.
More importantly, standard Fourier discretizations may lose monotonicity, leading to spurious value functions and unreliable policies \cite{du2025fourier}.

These challenges have motivated strictly monotone integration methods that enforce nonnegative conditional density representations \cite{zhang2024monotone}.
Like Fourier-based approaches, they leverage transform-domain information.
however, existing constructions typically rely on additional structure, for example, effectively 1D dynamics or simplified multi-dimensional models with closed-form nonnegative densities, and are available only in limited 2D settings \cite{zhou2025numerical,dang2025monotone}.

This paper develops a strictly monotone and provably convergent 2D integration framework for multi-period mean--CVaR optimization in fully coupled 2D models.
Our main contribution is to enforce strict monotonicity without requiring an explicit closed-form transition density, by learning a nonnegative 2D transition kernel directly from the increment law via its CF.
Building on the Fourier-domain density-learning idea in \cite{du2025fourier}, we approximate the kernel using a Gaussian-mixture representation fitted to the CF.
Constraining the mixture weights to the probability simplex ensures nonnegativity and normalization by construction.
The resulting 2D convolution integrals are discretized using composite quadrature rules with nonnegative weights, yielding a strictly monotone scheme that can be implemented efficiently using 2D FFTs.

A second contribution is a Fourier-to-real-space analysis that underpins convergence of the scheme. Leveraging the explicit Fourier form of both the Gaussian-mixture kernel and the CF, we derive Fourier-domain $L_2$ error estimates and translate them into real-space bounds. Under mild Fourier-tail decay assumptions,
these bounds are used to establish $\ell_\infty$-stability, consistency, and pointwise convergence as the discretization and kernel-approximation parameters vanish.

We keep the mean--CVaR formulation and modelling assumptions broad: between decision times the 2D increment law is specified via a closed-form CF with mild Fourier-tail decay. This setting covers a wide class of coupled 2D models where transition densities are unavailable but CFs are tractable.
To illustrate the method, we consider a multi-period portfolio optimization problem calibrated to long-horizon market data.
More generally, the proposed framework extends to other reward--risk stochastic control problems with discrete interventions and \mbox{translation-invariant kernels.}

\section{Modelling}
\label{sec:modelling}
We work on a filtered probability space $(\Omega,\mathcal{F},\{\mathcal{F}_t\}_{0 \le t \le T},\mathbb{P})$ over a finite horizon $T>0$.

\paragraph{Intervention times.}
We fix equally spaced decision times in $[0,T]$:
\begin{equation}
\label{eq:T_M}
\mathcal{T}=\{t_m \mid t_m=m\Delta t,\ m=0,\ldots,M\},\qquad \Delta t=T/M,
\end{equation}
with $t_0=0$. For any process $\{Z_t\}_{t\in [0, T]}$ with left/right limits at
$\{t_m\}$, we write
$Z_m^-:=\lim_{\varepsilon\to 0^+}Z(t_m-\varepsilon)$ and $Z_m^+:=\lim_{\varepsilon\to 0^+}Z(t_m+\varepsilon)$.

\paragraph{State variables and log representation.}
Let $A_t^s$ and $A_t^b$ denote two nonnegative state components at time $t\in[0,T]$ (e.g.\ two resource/inventory levels, two account values, or two components of a controlled system).

Intervention actions occur only at decision times. At each $t_m$, \mbox{$m=0,\ldots,M-1$,}
the state may be updated first by a deterministic exogenous input processed at $t_m^-$
(assumed nonnegative; e.g.\ an external injection/replenishment)
and then by an intervention applied at $t_m^+$. At $t_M=T$, no further intervention
is applied and a terminal reward functional is evaluated. Between decision times, i.e.\ on each interval $[t_m^+,t_{m+1}^-]$, no interventions are applied and the dynamics evolve according to the (uncontrolled) law specified below.

To work in log coordinates, we impose a strictly positive floor at decision times: after the contribution/intervention at $t_m$, we enforce
$A_{t_m^+}^s,A_{t_m^+}^b \ge e^{\winf}$ \mbox{for fixed $\winf\ll 0$.} Hence, for $t\in[t_m^+,t_{m+1}^-]$,
\[
S_t=\ln A_t^s,\qquad B_t=\ln A_t^b,\qquad X_t=(S_t,B_t)\in\Rbb^2.
\]
We write $x=(s,b)$ for a generic state, $y=(y_s,y_b)$ for an integration variable, and $\eta=(\eta_s,\eta_b)$ for its Fourier (frequency) counterpart.

\paragraph{Increment law and Fourier transform.}
On each inter-decision interval $[t_m^+,t_{m+1}^-]$, we define the one-step increment
\begin{equation}
\label{eq:DeltaX}
(\Delta X)_m \equiv ((\Delta S)_m,(\Delta B)_m):=X_{m+1}^- - X_m^+.
\end{equation}
We assume the conditional law of $(\Delta X)_m$ given $X_m^+$ depends only on $\Delta t$ and is translation-invariant; in particular, the law of $(\Delta X)_m$ does not depend on $m$. Let $\Delta X$ denote a generic one-step increment with this common law and define its CF
\[
G(\eta;\Delta t):=\mathbb{E}\!\left[e^{\,i\,\eta\cdot \Delta X}\right],\qquad \eta\in\Rbb^2.
\]
When $\Delta X$ admits a density $g(\cdot;\Delta t)$ on $\Rbb^2$, $g$ and $G$ form a Fourier pair:
\begin{equation}
\label{eq:FT_pair}
G(\eta;\Delta t)=\int_{\Rbb^2}e^{\,i\,\eta\cdot y}\,g(y;\Delta t)\,dy,\qquad
g(y;\Delta t)=\frac{1}{(2\pi)^2}\int_{\Rbb^2}e^{-i\,\eta\cdot y}\,G(\eta;\Delta t)\,d\eta.
\end{equation}

\begin{assumption}[Modelling assumptions]
\label{ass:modelling}
For each $\Delta t>0$, the uncontrolled one-step increment $\Delta X$ admits a density $g(\cdot;\Delta t)$ on $\Rbb^2$, and:
\begin{enumerate}[noitemsep, topsep=2pt,label=(A\arabic*), leftmargin=*]
\item (Homogeneous one-step kernel.)
For all $m$ and all $x\in\Rbb^2$, the conditional density of $X_{m+1}^-$ given $X_m^+=x$ is $y\mapsto g(y-x;\Delta t)$ (time-homogeneous and translation-invariant).

\item (Regularity and boundedness.)
$g(\cdot;\Delta t)\in L_1(\Rbb^2)\cap L_\infty(\Rbb^2)$.

\item (Closed-form CF.)
$G(\eta;\Delta t)$ is available in closed form for all $\eta\in\Rbb^2$.

\item (Fourier tail control.)
There exist $\alpha\in(0,2]$, $c_0>0$, and $R_{\mathrm{tail}}\ge 0$ such that
$|G(\eta;\Delta t)| \le \exp(-c_0\,\Delta t\,\|\eta\|_2^\alpha)$ for all $\|\eta\|_2\ge R_{\mathrm{tail}}$.

\item (Exponential moments.)
$G(-i,0;\Delta t)$, $G(0,-i;\Delta t)$, and $G(-i,-i;\Delta t)$ are finite, and $G(0,0;\Delta t)=1$.
\end{enumerate}
\end{assumption}

\begin{remark}[Examples and scope]
Assumption~\ref{ass:modelling} covers a broad class of translation invariant increment laws with tractable CFs.
In particular, it holds for many 2D exponential--L\'{e}vy and jump--diffusion models with closed-form characteristic functions, including correlated Brownian increments, Merton \cite{merton1976option} and Kou \cite{kou01} jump--diffusions with a nondegenerate Gaussian component, Normal--Inverse--Gaussian \cite{barndorff1997normal}, tempered-stable and CGMY models with $Y\in(0,2)$ \cite{carr2002fine}, and finite Gaussian mixtures \cite{mclachlan2019finite}, including variants with co-jumps.
Condition (A4) provides the decay needed for Fourier truncation, while (A5) ensures the exponential tilts used in
boundary propagation are admissible.
\end{remark}

\section{Mean--CVaR optimization}
This section specifies a two-component reward--risk stochastic control problem with discrete interventions.
The notation follows an allocation setting in which an aggregate level is redistributed between two nonnegative components.
Allocating inventory between two storage locations (or water between two reservoirs) is one concrete example.

\subsection{Intervention map}
\label{ssc:allocation}
Suppose the system is in state $x=(s,b)$ at time $t_m^-$ (immediately before the deterministic exogenous input $q_m\ge 0$), with aggregate level
$W_m^- = e^s + e^b$. Immediately after the input, at time $t_m^+$, the aggregate level is
\begin{equation}
\label{eq:cash}
W_m^+ \;=\; W_m^- + q_m \;=\; e^s+e^b+q_m,\qquad m = 0, \ldots M-1.
\end{equation}
An admissible Markov feedback control at $t_m$ is any Borel--measurable map
\EQ
\label{eq:control_def}
u_m:\ \Rbb^2\to \mathcal{Z},\qquad (s,b)\mapsto u_m(s,b).
\EN
Here, $\mathcal Z:=[0,1]$ is the set of admissible actions:
$u_m$ is interpreted as the fraction of $W_m^+$ assigned to the first component, and $1-u_m$ to the second (a simplex-type constraint).
For brevity, when the arguments are clear we write $u_m:=u_m(s,b)$.

After applying $u_m$, the post-intervention levels of the two components are
$u_m W_m^+$ and $(1-u_m)W_m^+$, respectively, with $W_m^+$ given by \eqref{eq:cash}.
To work in log coordinates, introduce a strictly positive floor $e^{\winf}$ with fixed
$\winf\ll 0$. The post--intervention log--state is $X_m^+=(S_m^+,B_m^+)$, where
\begin{equation}
\label{eq:state-update}
\begin{aligned}
S_m^+&:=s^{+}(s,b,q_m,u_m)
=\ln\bigl(\max\{\,u_m W_m^+,\,e^{\winf}\}\bigr),\\[2pt]
B_m^+&:=b^{+}(s,b,q_m,u_m)
=\ln\bigl(\max\{(1-u_m) W_m^+,\,e^{\winf}\}\bigr).
\end{aligned}
\end{equation}
An admissible policy is $\mathcal U_0=\{u_m\}_{m=0}^{M-1}$ with each $u_m$ as in
\eqref{eq:control_def}. We denote the set of such policies by $\mathcal A$, and write
$\mathcal U_m=\{u_{m'}\}_{m'=m}^{M-1}$ for the tail policy applied from $t_m$ onwards.

\subsection{Mean--CVaR formulation}
\label{ssc:formulation}
Although CVaR is often defined for losses, we apply it here to the terminal outcome
$W_T$, adopting the convention that larger values indicate better performance.
Equivalently, one may apply the standard loss-based CVaR to $-W_T$. Fix $\alpha \in (0,1)$. Following \cite{RT2000, miller2017optimal}, for any policy
$\mathcal{U}_0\in\mathcal A$, the CVaR of $W_T$ at level $\alpha$ admits the threshold representation
\begin{eqnarray}
\text{CVaR}_{\alpha,\,  \mathcal{U}_0}^{x_0,t_0^-}
&=&
  \sup_{w \,\ge 0}\,  \Ebb_{\mathcal{U}_0}^{x_0,\,t_{0}^-}
  \!\big[
      w ~+~
      \tfrac{1}{\alpha}\,\min\!\bigl(W_{T} - w,\,0\bigr)
    ~\big],
\label{eq: CVaR_optimization}
\end{eqnarray}
where the feasible set $w\in[0,\infty)$ matches the attainable range of $W_T$.

Using a scalarization parameter $\gamma>0$, we adopt a \emph{pre-commitment} mean--CVaR objective: at inception $t_0^-$ the decision maker selects a policy $\mathcal{U}_0\in\mathcal A$ and a CVaR threshold $w\ge 0$ to maximize the criterion, and then commits to $\mathcal{U}_0$ thereafter.
The pre-commitment mean--CVaR value function is
\begin{equation}
\label{eq:PCEC}
\mathcal{V}(x_0, t_0^{-})=
 \sup_{\mathcal{U}_0 \in \mathcal{A}}\, \sup_{w \ge 0}\,
 \Ebb_{\mathcal{U}_0}^{x_0,\,t_{0 }^{-}}
 \Big[ W_T + \gamma\Big(w+\tfrac{1}{\alpha} \min \left(W_T-w, 0\right)\Big)\Big].
\end{equation}
Since $\mathcal U_0$ and $w$ range over independent sets and $w$ enters only through the terminal objective functional, the iterated suprema commute:
\begin{equation}
\label{eq:PCEC_interchange}
\mathcal{V}(x_0, t_0^{-})=
 \sup_{w \ge 0}\, \sup_{\mathcal{U}_0 \in \mathcal{A}}\,
 \Ebb_{\mathcal{U}_0}^{x_0,\,t_{0 }^{-}}
 \Big[ W_T + \gamma\Big(w+\tfrac{1}{\alpha} \min \left(W_T-w, 0\right)\Big)\Big],
\end{equation}
subject to Assumption~\ref{ass:modelling} on modelling, the exogenous inputs
\eqref{eq:cash}, and the intervention map \eqref{eq:state-update}.
\begin{remark}[Finiteness of $\Ebb_{\mathcal{U}_0}^{x_0,t_0^-}\lbrack W_T\rbrack$ ]
\label{rmk:finiteness_W}
Under Assumption~\ref{ass:modelling}~(A5), we have \mbox{$\Ebb\!\left[e^{\Delta S}\right]=G(-i,0;\Delta t)<\infty$,}
and $\Ebb\!\left[e^{\Delta B}\right]=G(0,-i;\Delta t)<\infty$.
Since interventions are restricted to $u\in[0,1]$ and the exogenous inputs $\{q_m\}$ are deterministic and finite, these moment bounds propagate over the finite horizon to yield $\Ebb_{\mathcal U_0}^{x_0,t_0^-}[W_T]<\infty$ uniformly over admissible policies $\mathcal U_0\in\mathcal A$.
\end{remark}
\begin{lemma}[Existence of a finite optimal threshold]
\label{lem:CVaR_threshold_existence}
Fix $\gamma>0$. Then the outer optimization in \eqref{eq:PCEC_interchange} is finite and is
attained by some $w^*(x_0)\in[0,\infty)$.
\end{lemma}
A proof of Lemma~\ref{lem:CVaR_threshold_existence} is given in Appendix~\ref{app:CVaR_threshold_existence}.
In particular, Lemma~\ref{lem:CVaR_threshold_existence} (together with Remark~\ref{rmk:finiteness_W}) implies that the objective functional in \eqref{eq:PCEC}--\eqref{eq:PCEC_interchange} is finite (and hence well-defined).

\paragraph{Lifted-state recursion (fixed $w$).}
We reformulate \eqref{eq:PCEC_interchange} as an equivalent lifted-state problem
\cite{miller2017optimal}. For fixed $w\ge 0$, define the
auxiliary value on the augmented state $(x,w,t)$, $x=(s,b)$:
\begin{equation}
\label{eq:lifted_value_def}
V(x,w,t_m^-)\;:=\;\sup_{\mathcal U_m}\ \Ebb_{\mathcal U_m}^{x,t_m^-}\!\Big[
\,W_T+\gamma\big(w+\tfrac{1}{\alpha}\min(W_T-w,0)\big)\Big].
\end{equation}
Terminal condition: denoting the terminal reward functional by $\Phi(x,w)$, we have
\begin{equation}
\label{eq:terminal}
V(x,w,T^-)=\Phi(x,w):=(e^s+e^b)+\gamma\!\big(
w+\tfrac{1}{\alpha}\min\!\big((e^s+e^b)-w,0\big)\big),~ x=(s,b).
\end{equation}
Between $[t_m^+,t_{m+1}^-]$, $m=M-1,\ldots,0$, propagate via the uncontrolled kernel $g(\cdot;\Delta t)$:
\begin{equation}
\label{eq:integral}
V(x,w,t_m^+)\;=\;\int_{\Rbb^2} V(y,w,t_{m+1}^-)\; g(y-x;\Delta t)\,dy.
\end{equation}
At $t_m$, define the intervention operator for any test functional $F$
and control $u\in\mathcal Z$,
\begin{equation}
\label{eq:Operator_M}
(\mathcal M_u F)(x,w,t_m^-; q_m)\;:=\;F\big(s^+(x, q_m, u),\,b^+(x,q_m, u),\,w,\,t_m^+\big),
\end{equation}
where $(s^+(\cdot),b^+(\cdot))$ as in \eqref{eq:state-update}.
The optimal control (for fixed $w$) is
\begin{equation}
\label{eq:control}
u_m^\ast(\cdot;w) \in
\operatorname*{arg\,max}_{u\in\mathcal Z}\;
(\mathcal M_u V)(x,w,t_m^-;q_m).
\end{equation}
The pre-intervention auxiliary value at $t_m^-$ is
\begin{equation}
\label{eq:intervention}
V(x,w,t_m^-)\;=\;(\mathcal M_{u_m^\ast(\cdot;w)} V)(x,w,t_m^-; q_m).
\end{equation}
At inception, select the optimal threshold and value:
\begin{equation}
\label{eq:wstar_value}
w^\ast=\operatorname*{arg\,sup}_{w\ge 0}V(x_0,w,t_0^-),
\qquad
V(x_0,t_0^-)=V(x_0,w^\ast,t_0^-).
\end{equation}
The pre-commitment optimal policy is
$\mathcal{U}_0^{\ast}=\{u_0^{\ast}(\cdot\,;w^\ast),\ldots,u_{M-1}^{\ast}(\cdot\,;w^\ast)\}$.

\begin{proposition}[Equivalence to the pre-commitment problem]
\label{prop:equiv_lifted}
The lifted recursion \eqref{eq:lifted_value_def}–\eqref{eq:wstar_value} is equivalent to
\eqref{eq:PCEC_interchange}; see \cite{miller2017optimal}.
\end{proposition}

\subsection{Localization and problem statement}
\label{subsec:localization}
Since $w^{\ast}$ is finite by Lemma~\ref{lem:CVaR_threshold_existence}, we truncate the threshold domain to $\Gamma =[0,\,w_{\max}]$, where $w_{\max}>0$ is chosen sufficiently large. We localize the $(s,b)$–domain to the rectangle
\begin{equation}
\label{eq:omega}
\Omega \;=\; \big[s_{\min}^{\dagger},\, s_{\max}^{\dagger}\big]\times \big[b_{\min}^{\dagger},\, b_{\max}^{\dagger}\big],
\end{equation}
with \;$s_{\min}^{\dagger} < s_{\min} < 0 < s_{\max} < s_{\max}^{\dagger}$\; and \;$b_{\min}^{\dagger} < b_{\min} < 0 < b_{\max} < b_{\max}^{\dagger}$\; chosen so that boundary truncation errors are negligible (see, e.g.~\cite{zhang2024monotone, DangForsyth2014}).
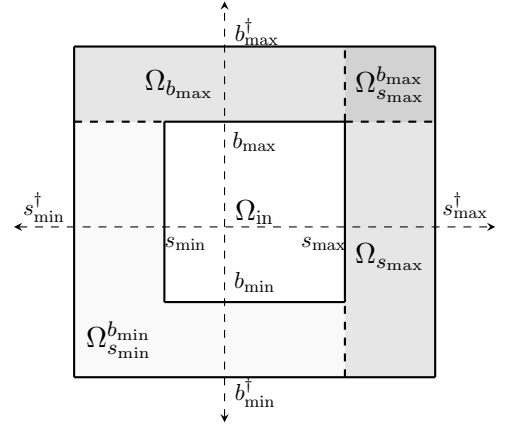
\begin{figure}[ht]
\centering
\begin{minipage}{0.58\textwidth}
We partition $\Omega$ into the interior and four boundary sub-domains:
\begin{equation}
\label{eq:omega_in}
\begin{aligned}
\Omega_{\myin} &= (s_{\min}, s_{\max}) \times (b_{\min}, b_{\max}),\\
\Omega_{s_{\max}} &= [s_{\max}, s_{\max}^{\dagger}] \times [b_{\min}^{\dagger}, b_{\max}),\\
\Omega_{s_{\max}}^{b_{\max}} &= [s_{\max},\,s_{\max}^{\dagger}] \times [b_{\max},\,b_{\max}^{\dagger}],\\
\Omega_{b_{\max}} &= [s_{\min}^{\dagger}, s_{\max}) \times [b_{\max}, b_{\max}^{\dagger}],\\
\Omega_{s_{\min}}^{b_{\min}} &= \Omega \setminus \Omega_{\myin} \setminus \Omega_{s_{\max}}
\setminus \Omega_{b_{\max}} \setminus \Omega_{s_{\max}}^{b_{\max}},\\
\Omega_{\myout} &:= \Omega \setminus \Omega_{\myin}.
\end{aligned}
\end{equation}
An illustration of the sub-domains for the localized problem
at each $t$ (fixed $w\in\Gamma$) is given in Figure~\ref{fig:domain}.
\end{minipage}
\hfill
\begin{minipage}{0.38\linewidth}
\centering
\begin{tikzpicture}[scale=0.4]
   \fill[gray!5] (-2,-3.5) rectangle (1,5);    
   \fill[gray!5] (-2,-3.5) rectangle (7,-1);   

   \fill[gray!20] (7,-3.5)  rectangle (10,5);

   \fill[gray!20] (-2,5)    rectangle (7,7.5);

   \fill[gray!35] (7,5)     rectangle (10,7.5);

   \draw [dashed, line width=0.4pt] [stealth-stealth] (-4,1.5) -- (12,1.5);
   \draw [dashed, line width=0.4pt] [stealth-stealth] (3,-5) -- (3,9);

   \draw [thick]  (1,5) -- (7,5);
   \draw [thick]  (1,-1) -- (7,-1);
   \draw [thick]  (1,-1) -- (1,5);
   \draw [thick]  (7,-1) -- (7,5);

   \draw [thick](-2,-3.5) -- (-2,7.5);
   \draw [thick](-2,7.5) -- (10,7.5);
   \draw [thick](10,-3.5) -- (10,7.5);
   \draw [thick](-2,-3.5) -- (10,-3.5);

   \draw [dashed, line width=0.8pt] (7,-3.5) -- (7,-1);
   \draw [dashed, line width=0.8pt] (7,5)   -- (7,7.5);
   \draw [dashed, line width=0.8pt] (-2,5)  -- (1,5);
   \draw [dashed, line width=0.8pt] (7,5)   -- (10,5);

   \node at (4,2) {\scalebox{1.0}{$\Omega_{\myin}$}};
   \node at (8.5,0.5)  {\scalebox{1.0}{$\Omega_{s_{\max}}$}};
   \node at (1.5,6.25) {\scalebox{1.0}{$\Omega_{b_{\max}}$}};
   \node at (-0.5,-2.25) {\scalebox{1.0}{$\Omega_{s_{\min}}^{b_{\min}}$}};
   \node at ( 8.5, 6.25) {\scalebox{1.0}{$\Omega_{s_{\max}}^{b_{\max}}$}};

   \node [below] at (-3,3)  {\scalebox{0.8}{$s_{\min}^{\dagger}$}};
   \node [below] at (11,3)  {\scalebox{0.8}{$s_{\max}^{\dagger}$}};
   \node [right] at (3,-4)  {\scalebox{0.8}{$b_{\min}^{\dagger}$}};
   \node [right] at (3,8)   {\scalebox{0.8}{$b_{\max}^{\dagger}$}};
   \node [below] at (1.7,1.5)  {\scalebox{0.8}{$s_{\min}$}};
   \node [below] at (6.3,1.5)  {\scalebox{0.8}{$s_{\max}$}};
   \node [above] at (4,-1)     {\scalebox{0.8}{$b_{\min}$}};
   \node [below] at (4,5)      {\scalebox{0.8}{$b_{\max}$}};
\end{tikzpicture}
\caption{Spatial sub-domains at each $t$ (fixed $w\in\Gamma$).}
\label{fig:domain}
\end{minipage}
\vspace*{-0.25cm}
\end{figure}
On $\Omega_{\myout}$, the terminal reward \eqref{eq:terminal} suggests exponential leading-order
forms ($e^s$, $e^b$, $e^{s+b}$, or a constant) depending on which boundary is approached.
The next lemma shows that these forms propagate over one time step via explicit CF factors.

\begin{lemma}[One–step propagation for exponential functions]
\label{lem:asympBC}
Let $\phi:\Rbb^2\times[0,T]\to\Rbb$.
Under Assumption~\ref{ass:modelling}\,(A5), for any
$ a \in\{(1,0),(0,1),(1,1),(0,0)\}$, if at $t_{m+1}^-$,
$\phi(x,t_{m+1}^-)=A(t_{m+1}^-)\,e^{\, a \cdot x}$,
where $A(t)$ denotes a generic unknown function of time,
then applying the convolution integral \eqref{eq:integral} with $\phi$ in place of $V$ yields
\[
\phi(x, t_m^+)=
\int_{\Rbb^2} \phi(y,t_{m+1}^-)\; g(y-x;\Delta t)\,dy
= G(-i\,a;\Delta t)\,\phi(x, t_{m+1}^-),
\]
where $G(-i\,a;\Delta t)\in(0,\infty)$ by Assumption~\ref{ass:modelling}\,(A5).
\end{lemma}
\begin{proof}
A change of variables $y=x+z$ gives $\int_{\Rbb^2} e^{a\cdot z} g(z;\Delta t)\,dz=G(-i\,a;\Delta t)$.
\end{proof}
We encode the boundary type by the selector $a:\Omega_{\myout}\to\{(1,0),(0,1),(1,1),(0,0)\}$,
\begin{equation}
\label{eq:ax}
a(x)=(1,0)\mathbf{1}_{\Omega_{s_{\max}}}(x)+(0,1)\mathbf{1}_{\Omega_{b_{\max}}}(x)+
(1,1)\mathbf{1}_{\Omega_{s_{\max}}^{b_{\max}}}(x)+(0,0)
\mathbf{1}_{\Omega_{s_{\min}}^{b_{\min}}}(x).
\end{equation}
Define the boundary operator
\begin{equation}
\label{eq:bd_operator}
(\mathcal{B}V)(x, w, t_m^+)
\;:=\;G\big(-i\,a(x);\Delta t\big)\;V(x, w, t_{m+1}^-),
\qquad x\in\Omega_{\myout}.
\end{equation}
No intervention is applied on $\Omega_{\myout}$, so we also carry values from $t_m^+$ to $t_m^-$ on the boundary. Hence, for $(x,t_m^{\pm}) \in \Omega_{\myout} \times\{t_m^{\pm}\}$,
\begin{equation}
\label{eq:boundary_carry}
V(x, w, t_m^-) = V(x, w, t_m^+)= (\mathcal B V)(x,w,t_m^+).
\end{equation}
On the interior $\Omega_{\myin}$,  \eqref{eq:integral} is approximated by the truncated 2-D convolution
\begin{equation}
\label{eq:integral_trunc_2d}
V\bigl(x, w, t_m^+\bigr)
~\simeq~
\int_{\Omega}V(y, w, t_{m+1}^-)\, g(y-x;\Delta t)\,dy,
\qquad x \in \Omega_{\myin},
\end{equation}
where $x=(s,b)$, $y=(y_s,y_b)$, and $dy\equiv dy_s\,dy_b$.

\begin{definition}[Localized Mean--CVaR formulation]
\label{def:mean_cvar_localized}
The localized value function at time $t_0^{-}$ is $V(x_0, w^{\ast}, t_0^-)$, where
$w^{\ast}$ is obtained by the outer search \eqref{eq:wstar_value}, and
$\mathcal{U}_0^{\ast}=\bigl\{u_0^{\ast}(\cdot\, ;w^{\ast}), \ldots, u_{M-1}^{\ast}(\cdot\, ; w^{\ast})\bigr\}$ is the associated optimal control.

For each fixed $w\in\Gamma$, the function $V(x, w, t)$ on $\Omega\times\mathcal T$ is specified by:
(i)~terminal condition \eqref{eq:terminal};
(ii)~boundary propagation \eqref{eq:boundary_carry} on $\Omega_{\myout}$ for $m=M-1,\ldots,0$;
(iii)~interior propagation \eqref{eq:integral_trunc_2d} on $\Omega_{\myin}$ for $m=M-1,\ldots,0$;
and (iv)~interior intervention \eqref{eq:control}--\eqref{eq:intervention} (with no intervention on $\Omega_{\myout}$).
\end{definition}

\begin{remark}[Uniqueness and regularity of the localized inner problem]
\label{rm:uniqueness}
For each fixed $w\in\Gamma$ and each $m=0,\ldots,M$, the localized Bellman recursion admits a unique bounded solution on $\Omega$; moreover $V(\cdot,w,t_m^-)$ is continuous in $(s,b)$ on $\Omega_{\myin}$ (cf.~\cite[Proposition~3.1]{zhang2024monotone}).
\end{remark}

\section{Fourier-trained transition kernel approximation}
\label{sec:fournet}
For notational simplicity, we suppress the explicit dependence on $\Delta t$ and write
$g(\cdot)\equiv g(\cdot;\Delta t):\Rbb^2\to\Rbb$ for the (uncontrolled) transition density,
with CF  $G(\cdot)\equiv G(\cdot;\Delta t)$.
We adopt the same convention for the {\myblue{kernel}} approximation $\widehat g(\cdot)$ and its {\myblue{CF}} $\widehat G(\cdot)$.

\subsection{Description}
We approximate $g(\cdot)$ by a single--hidden--layer FFNN with Gaussian activation
\cite{du2025fourier}. In {\myblue{2D}} it is convenient to parameterize the network as a
finite Gaussian mixture. Specifically, for $y=(y_s,y_b)\in\Rbb^2$ we take
\begin{equation}
\label{eq:act_func}
\phi(y;\mu,\Sigma)
:= \frac{1}{(2\pi)\,|\Sigma|^{1/2}}\,
\exp\!\Big(-\tfrac12\,(y-\mu)^\top \Sigma^{-1}(y-\mu)\Big),
\qquad \mu\in\Rbb^2,\ \ \Sigma\in\Rbb^{2\times 2},
\end{equation}
and define the class of single-layer FFNNs with Gaussian activation by
\begin{equation}
\label{eq:Sigma_class}
\Sigma(\phi)
=\Big\{
\widehat{g}:\Rbb^2\to\Rbb\ \Big|\ \
\widehat{g}(y;\theta)=\sum_{n=1}^{N}\beta_n\,\phi(y;\mu_n,\Sigma_n),\ \theta\in\Theta,\ N\in\Nbb
\Big\},
\end{equation}
where
\begin{equation}
\label{eq:mu_Sigma}
\mu_n=(\mu_n^s,\,\mu_n^b)\in\Rbb^2,\qquad
\Sigma_n=\begin{pmatrix}
(\sigma^s_n)^2 & \rho_n\,\sigma^s_n\sigma^b_n\\[2pt]
\rho_n\,\sigma^s_n\sigma^b_n & (\sigma^b_n)^2
\end{pmatrix},
\end{equation}
with $\sigma^s_n>0$, $\sigma^b_n>0$, and $\rho_n\in(-1,1)$, so each $\Sigma_n$ is symmetric
positive--definite.

\paragraph{Bounded parameter space and positivity.}
By Assumptions~\ref{ass:modelling}\,(A1)--(A2), $g(\cdot;\Delta t)\in L_1(\Rbb^2)\cap L_2(\Rbb^2)$.
The approximation results in \cite[Cor.~3.2, Thm.~3.4]{du2025fourier} imply that
single--hidden--layer Gaussian networks are $L_2$--dense on $L_1\cap L_2$ even with bounded parameters.
We therefore restrict to the bounded parameter set
\begin{equation}
\label{eq:Theta_bounds}
\begin{aligned}
\Theta
=\big\{
\theta &=\{(\beta_n,\mu_n^s,\mu_n^b,\sigma^s_n,\sigma^b_n,\rho_n)\}_{n=1}^{N}:\
\beta_n\ge 0,\ \ \textstyle\sum_{n=1}^N \beta_n = 1, \\
&\qquad |\mu_n^s|,\,|\mu_n^b|\le \overline\mu,\quad
0< \sigma_{\min}\le \sigma^s_n,\sigma^b_n\le \sigma_{\max},\quad
|\rho_n|\le \overline\rho<1
\big\},
\end{aligned}
\end{equation}
for fixed finite $\overline\mu,\sigma_{\max}$, positive $\sigma_{\min}$, and $0<\overline\rho<1$.
Under \eqref{eq:Theta_bounds}, each $\Sigma_n$ is uniformly positive--definite. Moreover,
since $\phi(\cdot;\mu_n,\Sigma_n)\ge 0$ integrates to one and $\{\beta_n\}$ lies on the probability simplex,
$\widehat g(\cdot;\theta)$ is a proper density: $\widehat g\ge 0$ and $\int_{\Rbb^2}\widehat g=1$.

\paragraph{{\myblue{Closed--form CF.}}}
Using the Fourier pair in \eqref{eq:FT_pair}, the {\myblue{CF}} of $\widehat g(\cdot)$ is available in closed form.
For $\eta=(\eta_s,\eta_b)\in\Rbb^2$,
\begin{equation}
\label{eq:Gb_closed_form}
\widehat{G}(\eta;\theta)
=\int_{\Rbb^2}e^{\,i\,\eta\cdot y}\, \widehat{g}(y;\theta)\,dy
=\sum_{n=1}^{N}\beta_n\,\exp\!\Big(
i\,\eta^\top \mu_n-\tfrac12\, \eta^\top \Sigma_n \eta \Big).
\end{equation}

\paragraph{$L_2$-approximation error and Fourier-domain invariance.}
By Assumption~\ref{ass:modelling}\,(A2), $g(\cdot)\in L_1(\Rbb^2)\cap L_\infty(\Rbb^2)$ and hence $g(\cdot)\in L_2(\Rbb^2)$.
The Fourier transform is an $L_2$ isometry up to a constant factor (see, e.g.\
\cite{yosida1968functional}):
if $f\in L_2(\Rbb^2)$ with Fourier transform $\mathfrak{F}[f]$ under the convention \eqref{eq:FT_pair}, then
\begin{equation}
\label{eq:Fourier_L2_isometry}
\int_{\Rbb^2}|f(y)|^2\,dy \;=\; \frac{1}{(2\pi)^2}\int_{\Rbb^2}|\mathfrak{F} [f](\eta)|^2\,d\eta.
\end{equation}
Combining \eqref{eq:Fourier_L2_isometry} with \cite[Thm.~3.4]{du2025fourier} yields the {\myblue{2D}} approximation result below.
\begin{theorem}[$L_2$-approximation and Fourier-domain invariance; cf.\ \cite{du2025fourier}, Thm.~3.4]
\label{thm:fournet_existence}
For any $\varepsilon>0$, there exists $\widehat{g}(\cdot;\theta^\star_\varepsilon)\in\Sigma(\phi)$
with $\theta^\star_\varepsilon\in\Theta$ such that
\begin{equation}
\label{eq:plancherel_twoD}
\int_{\Rbb^2}\!\big|\,g(y)-\widehat{g}(y;\theta^\star_\varepsilon)\,\big|^2\,dy
\;=\;\frac{1}{(2\pi)^2}\int_{\Rbb^2}\!\big|\,G(\eta)-\widehat{G}(\eta;\theta^\star_\varepsilon)\,\big|^2\,d\eta
\;<\;\varepsilon,
\end{equation}
where $\widehat{G}(\eta;\theta)$ is given in \eqref{eq:Gb_closed_form}.
\end{theorem}
\begin{proof}
Existence follows by a {\myblue{2D}} adaptation of \cite[Thm.~3.4]{du2025fourier} for $\Sigma(\phi)$ with bounded parameters $\Theta$. The equality in \eqref{eq:plancherel_twoD} follows from
\eqref{eq:Fourier_L2_isometry} applied to $f=g-\widehat g$.
\end{proof}

\subsection{Training loss and regularization}
\label{subsec:fournet_loss_error}
We truncate the Fourier domain from $\Rbb^2$ to $D_\eta=[-\eta',\eta']^2$,
with $\eta'>0$ chosen sufficiently large, and
discretize $D_\eta$ by a (possibly non-uniform) set of nodes $\{\eta_p\}_{p=1}^P\subset D_\eta$, $\eta_p=(\eta_{p}^s,\eta_{p}^b)$.
Define
\[
\delta_{\min}:=\min_{p\ne q}\|\eta_p-\eta_q\|_2,
\qquad
\delta_{\max}:=\max_{1\le p\le P}\min_{q\ne p}\|\eta_p-\eta_q\|_2,
\]
and assume the sampling set is quasi-uniform: there exist constants $C_0,C_1>0$,
independent of $P$ and $N$, such that
\begin{equation}
\label{eq:delta_eta}
C_0\,P^{-1/2}\;\le\;\delta_{\min}\;\le\;\delta_{\max}\;\le\;C_1\,P^{-1/2}.
\end{equation}
Let $\widehat{\Theta}\subset\Theta$ be the empirical parameter set. Using the closed-form target CF
$G(\cdot)$ and the network CF $\widehat{G}(\cdot;\theta)$ from \eqref{eq:Gb_closed_form},
we compute $\theta$ by minimizing the empirical loss
\begin{equation}
\label{eq:loss}
\mathrm{Loss}_P(\theta)
=\frac{1}{P}\sum_{p=1}^{P}\big|\,G(\eta_p)-\widehat{G}(\eta_p;\theta)\,\big|^2
\;+\; R_P(\theta),
\qquad \theta\in\widehat{\Theta}.
\end{equation}
Here, $R_P(\theta)$ is the MAE regularization term
\[
R_P(\theta)=\frac{1}{P}\sum_{p=1}^{P}\Big(
\big|\text{Re}_G(\eta_p)-\text{Re}_{\widehat{G}}(\eta_p;\theta)\big|
+\big|\text{Im}_G(\eta_p)-\text{Im}_{\widehat{G}}(\eta_p;\theta)\big|
\Big),
\]
where $\text{Re}_G,\text{Im}_G$ denote the real and imaginary parts of $G$ (and likewise for $\widehat G$).
The empirical minimizer is
\begin{equation}
\label{eq:thetastar}
\widehat{\theta}^{\star} \;=\;
\operatorname*{arg\,min}_{\theta \in \widehat{\Theta}} \mathrm{Loss}_P(\theta).
\end{equation}
The MAE term concentrates accuracy in high-impact regions of $G$ (e.g.\ near peaks and rapid variations) \cite{du2025fourier}. In practice, the sampling set $\{\eta_p\}$ is chosen to be denser in such regions.

\subsection{Error decomposition}
\label{subsec:fournet_error}
We decompose the learned-kernel error into: (i) Fourier truncation on $\Rbb^2\setminus D_\eta$, (ii) empirical training error on the sampled set, and (iii) sampling (quadrature) error on $D_\eta$. The next lemma provides explicit bounds
\mbox{for $g-\widehat g$.}

\begin{lemma}[2D training error decomposition]
\label{lem:fournet_L2_bound}
Let $D_\eta=[-\eta',\eta']^2$ and let $\{\eta_p\}_{p=1}^P\subset D_\eta$ satisfy
\eqref{eq:delta_eta}. Let $\widehat{\theta}^\star\in\widehat{\Theta}$ be the empirical minimizer in
\eqref{eq:thetastar}. Assume that, for some $\varepsilon_1>0$ and
$f\in\{G(\cdot),\,\widehat{G}(\cdot;\widehat{\theta}^\star)\}$,
\begin{equation}
\label{eq:boundD}
\int_{\Rbb^2\setminus D_\eta}\big(|\mathrm{Re}_f(\eta)|+|\mathrm{Im}_f(\eta)|\big)\,d\eta<\varepsilon_1,
\qquad
\int_{\Rbb^2\setminus D_\eta}|f(\eta)|^2\,d\eta<\varepsilon_1,
\end{equation}
and that $\mathrm{Loss}_P(\widehat{\theta}^\star)<\varepsilon_2$
and $R_P(\widehat{\theta}^\star)<\varepsilon_3$. Moreover, suppose
\begin{equation}
\label{eq:grad}
C' ~:=~ \sup_{\eta\in D_\eta,\;\theta\in\Theta}
\big\|\nabla_{\eta}\,|G(\eta)-\widehat{G}(\eta;\theta)|^2\big\|_2
~<~\infty.
\end{equation}
Then:

\noindent\emph{(i) Global $L_2$-error bound.}
\begin{equation}
\label{eq:fournet_L2_bound}
\int_{\Rbb^2}\big|g(y)-\widehat{g}(y;\widehat{\theta}^\star)\big|^2\,dy
\;\le\;\frac{1}{(2\pi)^2}\Big\{4\varepsilon_1 + C_1\varepsilon_2
+ \frac{C'\,C_1^{3}}{2\,P^{1/2}}\,\Big\}.
\end{equation}

\noindent\emph{(ii) Pointwise bound.}
For every $y\in\Rbb^2$,
\begin{equation}
\label{eq:Linfty_and_nonneg}
\big|g(y)-\widehat{g}(y;\widehat{\theta}^\star)\big|
\;\le\;\frac{1}{(2\pi)^2}\Big\{\,2\varepsilon_1 + C_1\,\varepsilon_3
+ \frac{C'\,C_1^{3}}{2\,P^{1/2}}\,\Big\}.
\end{equation}
\end{lemma}
A proof of Lemma~\ref{lem:fournet_L2_bound} is given in Appendix~\ref{app:fournet_L2_bound}.

We now show that, under Assumption~\ref{ass:modelling}\,(A4) and the bounded parameter set \eqref{eq:Theta_bounds}, the Fourier mass outside the truncated domain $D_\eta=[-\eta',\eta']^2$ (and hence the Fourier truncation error) decays exponentially for both $G$ and $\widehat G$.
\begin{lemma}[Fourier–domain truncation error bounds]
\label{lem:fourier_truncation_tails}
Suppose Assumption~\ref{ass:modelling} (A4) holds for $G(\cdot)$.
Recall $\widehat G(\cdot;\theta)$ from \eqref{eq:Gb_closed_form} with parameters bounded as in \eqref{eq:Theta_bounds}.
There exist positive constants $C_G^{(p)},c_G^{(p)}$, and, for each fixed finite $N$, constants $C_{\widehat G}^{(p)}$, $c_{\widehat G}^{(p)}>0$
(uniform over all $\theta$ in the bounded parameter set with $N$ terms) such that, for any $\eta'\ge R_{\mathrm{tail}}$ and $p\in\{1,2\}$,
\[
\int_{\Rbb^2\setminus[-\eta',\eta']^2}\!\!\! |G(\eta)|^p\,d\eta \le C_G^{(p)}\,e^{-\,c_G^{(p)}\,(\eta')^\alpha},
\quad
\int_{\Rbb^2\setminus[-\eta',\eta']^2} \!\!\!|\widehat G(\eta;\theta)|^p\,d\eta \le C_{\widehat G}^{(p)}\,e^{-\,c_{\widehat G}^{(p)}\,(\eta')^{2}}.
\]
\end{lemma}
A proof of Lemma~\ref{lem:fourier_truncation_tails} is given in Appendix~\ref{app:fourier_tails}.

Let $h \in (0, 1)$ denote a target tolerance.
The next corollary gives a sufficient choice of $\eta'=\eta'(h)$
so that the Fourier–domain truncation error for both $G$ and $\widehat G$
satisfies $\le h^{\,1+\kappa}$ for any $\kappa\in(0,1]$.
\begin{corollary}[Fourier-domain truncation error $O(h^{1+\kappa})$]
\label{cor:tail_Ch2}
Fix $p\in\{1,2\}$ and $\kappa\in(0,1]$. Under Assumption~\ref{ass:modelling}\,(A4)
and Lemma~\ref{lem:fourier_truncation_tails}, define for $h\in(0,1)$
\[
\eta'(h)\;:=\;\max\!\bigg\{
R_{\mathrm{tail}},\;
\Big(\tfrac{1}{c_G^{(p)}}\log\tfrac{2C_G^{(p)}}{h^{\,1+\kappa}}\Big)^{\!1/\alpha},
\;
\Big(\tfrac{1}{c_{\widehat G}}\log\tfrac{2C_{\widehat G}^{(p)}}{h^{\,1+\kappa}}\Big)^{\!1/2}
\bigg\}.
\]
Then
$\ds \max_{f\in\{G,\,\widehat G\}}
\int_{\Rbb^2\setminus[-\eta'(h),\,\eta'(h)]^2} |f(\eta)|^p\,d\eta
\;\le\; h^{\,1+\kappa}$.
\end{corollary}

To establish our error bounds in terms of $h$, we introduce the following sampling and training assumptions, indexed by a target tolerance $h\in(0,1]$.
\begin{assumption}[Sampling and training assumptions]
\label{ass:training_regime}
Let $h\in(0,1)$ be a training tolerance parameter. Fix $\kappa\in(0,1)$ and choose $\eta'(h)$ as in Corollary~\ref{cor:tail_Ch2} so that, for $p\in\{1,2\}$,
\[
\max_{f\in\{G,\widehat G\}}\int_{\Rbb^2\setminus[-\eta',\eta']^2} |f(\eta)|^p\,d\eta
\le c_1\,h^{\,1+\kappa},
\qquad \eta'=\eta'(h).
\]
Assume the quasi-uniformity constant $C_1$ in \eqref{eq:delta_eta} and the gradient bound $C'$ in \eqref{eq:grad} are finite and independent of $h$.
Assume further that the trained network satisfies:
(i) $\mathrm{Loss}_P(\widehat\theta^\star)\le c_2\,h^{\,1+\kappa}$,
(ii) $R_P(\widehat\theta^\star)\le c_3\,h^{\,1+\kappa}$,
and (iii) $P\ge c_P\,h^{-2(1+\kappa)}$,
where $c_1,c_2,c_3,c_P>0$ are constants independent of $h$.
\end{assumption}

\begin{corollary}[$h$–dependent convergence rates]
\label{cor:fournet_rates}
Under Assumption~\ref{ass:training_regime}, there exists a generic constant $C>0$,
independent of $h$ (its value may change from line to line), such that:

\noindent\emph{(i) Global $L_2$-error bound.}
$
\|g-\widehat g\|_{L_2(\Rbb^2)} ~\le~ C\,h^{\,\frac{1+\kappa}{2}}.
$

\noindent\emph{(ii) Pointwise bound.}
For every $y\in\Rbb^2$,
\EQ
\label{eq:pointwise_g}
|g(y)-\widehat g(y;\widehat\theta^\star)|
~\le~ C\,h^{\,1+\kappa}.
\EN
\end{corollary}
A proof of Corollary~\ref{cor:fournet_rates} is given in Appendix~\ref{app:fournet_rates}.

\begin{remark}[Model capacity is implicit]
\label{rmk:N_of_h}
Assumption~\ref{ass:training_regime} does not prescribe a specific growth law $N(h)$ for the network size.
We only require that (for each $h$) there exist trained parameters (possibly with $N=N(h)$) that achieve the stated empirical error bounds.
Thus, any architecture-size schedule that delivers $\mathrm{Loss}_P(\widehat{\theta}^\star)=O(h^{1+\kappa})$ and $R_P(\widehat{\theta}^\star)=O(h^{1+\kappa})$ is admissible, and we keep $N(h)$ implicit.
\end{remark}

\section{Numerical methods}
\label{sec:num_methods}

\subsection{Discretization}
To compute the interior convolution on $\Omega_{\myin}$ and the boundary propagation on
$\Omega\setminus\Omega_{\myin}$ on a single tensor grid, we pad each interior interval
$[z_{\min},z_{\max}]$ ($z\in\{s,b\}$) symmetrically by half its width (chosen sufficiently large so that boundary truncation errors are negligible; cf.~Section~\ref{subsec:localization}):
for $z\in\{s,b\}$,
\begin{equation}
\label{eq:pad_bounds}
z_{\min}^\dagger \;=\; z_{\min}-\tfrac12\,(z_{\max}-z_{\min}),\qquad
z_{\max}^\dagger \;=\; z_{\max}+\tfrac12\,(z_{\max}-z_{\min}).
\end{equation}
Let $D_z:=z_{\max}-z_{\min}$ and $D_z^\dagger:=z_{\max}^\dagger-z_{\min}^\dagger=2D_z$.
For $z\in\{s,b\}$, let $N_z$ be the number of intervals on $[z_{\min},z_{\max}]$ and set
$N_z^\dagger:=2N_z$ on $[z_{\min}^\dagger,z_{\max}^\dagger]$, so that
\begin{equation}
\label{eq:single_mesh}
\Delta z \;=\; \frac{D_z}{N_z} \;=\; \frac{D_z^\dagger}{N_z^\dagger},\qquad z\in\{s,b\}.
\end{equation}
Assume $N_s$ and $N_b$ are even and define the grid nodes by
\begin{equation}
\label{eq:grid_nodes}
s_k \;=\; \hat{s}_0 + k\,\Delta s,\quad k=-\tfrac{N_s^\dagger}{2},\ldots,\tfrac{N_s^\dagger}{2},
\qquad
b_j \;=\; \hat{b}_0 + j\,\Delta b,\quad j=-\tfrac{N_b^\dagger}{2},\ldots,\tfrac{N_b^\dagger}{2},
\end{equation}
where $\hat{s}_0 := \tfrac12(s_{\min}^\dagger + s_{\max}^\dagger)$ and
$\hat{b}_0 := \tfrac12(b_{\min}^\dagger + b_{\max}^\dagger)$.
For $z\in\{s,b\}$, define the global and interior index sets
\[
\mathbb{N}_z := \{-\tfrac{N_z}{2}+1,\,\ldots,\,\tfrac{N_z}{2}-1\}
\quad
\subset
\quad
\mathbb{N}_z^\dagger := \{-\tfrac{N_z^\dagger}{2},\,\ldots,\,\tfrac{N_z^\dagger}{2}\},
\]
so both the interior and boundary updates use the same global node set $\{(s_k,b_j)\}$.

We discretize the admissible {\myblue{control}} actions $\mathcal Z=[0,1]$ using $N_u$ uniform
intervals, yielding nodes $\{u_{\iota}\}_{\iota=0}^{N_u}$, and discretize the threshold domain
$\Gamma=[0,w_{\max}]$ at uniform nodes $\{w_c\}_{c=0}^{N_w}$, with spacings $\Delta u$ and $\Delta w$.

\paragraph{Single refinement parameter.}
As is common in the literature (e.g., \cite{MaForsyth2015, chen08a}[Eq.~(4.1)], \cite{dang2025monotone}[Eq.~(3.24)]),
we introduce a single parameter $h\in(0,1)$ to index both mesh refinement and the training tolerance from Assumption~\ref{ass:training_regime}.
For the spatial, control, and threshold grids we set
\begin{equation}
\label{eq:dis_parameter}
\Delta s = C_1 h,\qquad
\Delta b = C_2 h,\qquad
\Delta u = C_3 h,\qquad
\Delta w = C_4 h,
\end{equation}
with $C_1,C_2,C_3,C_4>0$ independent of $h$. We also index the Fourier truncation and sampling budget by $h$ via $\eta'=\eta'(h)$ (Corollary~\ref{cor:tail_Ch2}) and $P=P(h)$ with $P(h)\ge c_P\,h^{-2(1+\kappa)}$ (Assumption~\ref{ass:training_regime}).

For each discretized threshold value $w_c$, let $V_{k,j,c}^{m,\circ}$ be a numerical approximation to
$V(s_k,b_j,w_c,t_m^\circ)$ at the node $(s_k,b_j,w_c,t_m^\circ)$, where $t_m^\circ\in\{t_m,\,t_m^\pm\}$.
Given nodal values $\{V_{l,d,c}^{m,+}\}$ on the global grid, we denote by
$\mathcal{I}\big[\{V_c^{m,+}\}\big](s,b)$ the bilinear interpolant on $\Omega$ (for this fixed $w_c$).

Because the convolution integral \eqref{eq:integral_trunc_2d} evaluates the kernel at $(y-x)$, we use the index-difference notation
\begin{equation}
\label{eq:g_hatg}
g_{\,l-k,\,d-j}:=g \big((s_l-s_k,\;b_d-b_j);\Delta t\big),
\qquad
\widehat g_{\,l-k,\,d-j}:=\widehat g\big((s_l-s_k,\;b_d-b_j);\widehat{\theta}^{\star}\big),
\end{equation}
for $(k,j)\in\mathbb{N}_s\times\mathbb{N}_b$ and $(l,d)\in\mathbb{N}_s^\dagger\times\mathbb{N}_b^\dagger$,
where $\widehat{\theta}^{\star}$ is the minimizer \eqref{eq:thetastar}. Since $\widehat g(\cdot;\widehat{\theta}^{\star})$ is available in closed form on $\mathbb{R}^2$ (Section~\ref{sec:fournet}), these evaluations are well-defined.

\subsection{Numerical scheme}
\label{ssc:num_scheme}
We compute the lifted value function backward in time for each fixed discretized threshold $w_c\in\{w_c\}_{c=0}^{N_w}$.

\paragraph{Terminal condition.}
At maturity $t_M=T$, we implement \eqref{eq:terminal} on $\Omega$ via
\begin{equation}
\label{eq:num_terminal}
V_{k,j, c}^{M,-}
=
\Phi(s_k, b_j,w_c),
\quad (k,j)\in\mathbb N_s^\dagger\times\mathbb N_b^\dagger.
\end{equation}

\paragraph{Boundary sub-domains.}
For nodes $(s_k, b_j, t_m^{\pm})$ on $\Omega_{\myout} \times\{t_m^{\pm}\}$, $m = M-1, \ldots, 0$,
we enforce the asymptotic boundary conditions \eqref{eq:boundary_carry} as follows:
\begin{equation}
\label{eq:num_bd_out}
V_{k,j, c}^{m,-} = V_{k,j, c}^{m,+} = G\big(-i\,a(s_k, b_j);\Delta t\big)\; V_{k,j, c}^{m+1,-}.
\end{equation}

\paragraph{Interior sub-domain.}
For nodes $(s_k, b_j, t_m^+) \in \Omega_{\myin} \times \{t_m^+\}$, $m = M-1, \ldots, 0$,
the truncated 2-D convolution integral \eqref{eq:integral_trunc_2d} is approximated by
\begin{equation}
\label{eq:num_interior}
V_{k,j, c}^{m,+}
=
\Delta s\,\Delta b
\sum_{l\in\mathbb N_s^\dagger}\sum_{d\in\mathbb N_b^\dagger}
\varphi_{l,d}\;\widehat g_{l-k,\,d-j}\; V_{l,d, c}^{m+1,-},
\qquad (k,j)\in\mathbb N_s\times\mathbb N_b.
\end{equation}
Here, $\{\varphi_{l,d}\}$ are the composite trapezoidal weights (unit weight in the interior, $\tfrac12$ on edges, and $\tfrac14$ at corners).
The kernel weights $\widehat g_{l-k,d-j}$ are evaluated directly from the learned Gaussian-mixture representation in Section~\ref{sec:fournet}
(via \eqref{eq:act_func} and \eqref{eq:Sigma_class} with $\theta^*$).

{\myblue{\paragraph{Intervention (interior only).}}}
For each interior node $(s_k,b_j)\in\Omega_{\myin}$ at decision time $t_m$, $m=M-1,\ldots,0$,
we enforce \eqref{eq:control}–\eqref{eq:intervention} on the discrete control set $\{u_\iota\}_{\iota=0}^{N_u}$:
\begin{equation}
\begin{aligned}
V_{k,j, c}^{m,-}
&=
\max_{u_\iota \in \{u_\iota\}_{\iota= 0}^{N_u}}
\mathcal{I} \big[\big\{V_c^{m,+}\big\}\big]
\big(s_{k,j}^{m, +}(u_{\iota}),\,b_{k,j}^{m, +}(u_{\iota})\big),
\qquad (k,j)\in\mathbb N_s\times\mathbb N_b,\\
&s_{k,j}^{m, +}(u)=s^+(s_k,b_j, q_m, u),\qquad
b_{k,j}^{m, +}(u)=b^+(s_k,b_j, q_m, u),
\end{aligned}
\label{eq:num_intervention}
\end{equation}
where $(s^+(\cdot),b^+(\cdot))$ are given by \eqref{eq:state-update} and $\mathcal{I}[\cdot]$ denotes bilinear interpolation.
{\myblue{This step yields the numerically computed optimal control}} $u_{k,j,c}^{m,\ast}\equiv u_{k,j}^{m,\ast}(w_c)$:
\begin{equation}
\label{eq:u_star_discrete}
u_{k,j, c}^{m,\ast} \equiv  u_{k,j}^{m, \ast}(w_c)
~\in~
\operatorname*{arg\,max}_{u_\iota \in \{u_\iota\}_{\iota= 0}^{N_u}}
\mathcal{I} \big[\big\{V_c^{m,+}\big\}\big]
\big(s_{k,j}^{m, +}(u_{\iota}),\,b_{k,j}^{m, +}(u_{\iota})\big).
\end{equation}

\paragraph{Initial time $t_0$.}
At $t_0$, the post-intervention values $\{V_c^{0,-}\}$ are available on the grid.
Define $\widehat{V}(s,b,w_c,t_0^-):=\mathcal{I}[\{V_c^{0,-}\}](s,b)$ and determine the pre-commitment threshold and value by searching over $\{w_c\}_{c=0}^{N_w}$:
\begin{equation}
\label{eq:W_star_discrete}
w_{c\ast} \in \operatorname*{arg\,max}_{\{w_c\}_{c = 0}^{N_w}} \widehat{V}(s_0,b_0,w_c,t_0^-),
\qquad
V_h(s_0,b_0,t_0^-)\;:=\;\widehat{V}(s_0,b_0,w_{c\ast},t_0^-).
\end{equation}
Here, $V_h(s_0,b_0,t_0^-)$ is the numerical approximation (at refinement parameter $h$) to $V(s_0,b_0,t_0^-)$; the corresponding optimal controls $u_{k,j,c\ast}^{m,\ast}$ are obtained and stored during the backward recursion.

\subsection{Efficient implementation}
We accelerate the 2-D discrete convolution \eqref{eq:num_interior} using FFTs. Since $\widehat g_{l-k,d-j}$ depends only on index differences, the discrete convolution operator is Toeplitz-block-Toeplitz and can be embedded into a 2-D circular convolution on an augmented grid \cite{dang2025monotone,zhou2025numerical}. Concretely, with $\ast$ denoting circular convolution, \eqref{eq:num_interior} can be written in the circular-convolution form
\begin{equation}
\label{eq:cir_p}
{\bf{V}}^{m,+}_c =  \Delta s \Delta b~ {\bf{\widehat{g}}} \ast {\bf{V}}^{m+1,-}_c.
\end{equation}
Here,  ${\bf{\widehat{g}}}$ and ${\bf{V}}^{m+1,-}_c$ respectively denote the appropriately augmented (zero-padded and re-indexed) kernel and value arrays of size $(3N_s-1)\times(3N_b-1)$ constructed from $\{\widehat g_{l-k,\,d-j}\}$ and $\{\varphi_{l,d}V_{l,d,c}^{m+1,-}\}_{l\in\mathbb N_s^\dagger,\,d\in\mathbb N_b^\dagger}$ (see \cite{dang2025monotone, zhou2025numerical}). The notation $\ast$ denotes circular convolution.
The circular convolution \eqref{eq:cir_p} is then computed via FFT/iFFT:
\EQA
\label{eq:fft_ifft}
{\bf{V}}^{m,+}_c = \Delta s \Delta b~  {\text{iFFT}}\l\{\text{FFT}\l\{{\bf{V}}^{m+1,-}_c\r\} \circ \text{FFT}\l\{{\bf{\widehat{g}}}\r\}\r\}.
\ENA
After \eqref{eq:fft_ifft}, we extract the components corresponding to $(k,j)\in\mathbb N_s\times\mathbb N_b$ to obtain the interior values $V_{k,j,c}^{m,+}$ on $\Omega_{\myin}\times\{t_m^+\}$.

\section{Convergence analysis}
\label{sc:conv}
In this section, we use a Barles--Souganidis--type analysis \cite{barles-souganidis:1991}
to prove convergence, as $h\to 0$, to the localized formulation in
Definition~\ref{def:mean_cvar_localized}. Fix a discrete threshold node $w_c\in\Gamma_h$.
We verify $\ell_\infty$-stability, monotonicity, and pointwise consistency for the associated
inner recursion, which applies \eqref{eq:num_terminal} on $\Omega\times\{T\}$,
\eqref{eq:num_bd_out} on $\Omega_{\myout}$ (updated at $t_m^+$ and carried to $t_m^-$),
and \eqref{eq:num_interior}--\eqref{eq:num_intervention} on $\Omega_{\myin}\times\{t_m^\pm\}$,
for $m=M-1,\ldots,0$.

For an arbitrary $\theta\in\widehat{\Theta}$, write
$\widehat g^{(\theta)}(y):=\widehat g(y;\theta)$ and
$\widehat g^{(\theta)}_{\,l-k,\,d-j}:=\widehat g\big((s_l-s_k,\;b_d-b_j);\theta\big)$.
When $\theta=\widehat\theta^\star$, we drop $\widehat\theta^\star$
and revert to the convention adopted in \eqref{eq:g_hatg}.

\paragraph{A discrete kernel-mass bound.}
Assumption~\ref{ass:modelling}\,(A4) implies $g(\cdot;\Delta t)\in C^\infty(\Rbb^2)$ with
bounded derivatives \cite[Ch.~3]{stein2011fourier}. In particular, for any multi-index $\beta$,
\begin{align*}
\partial^\beta g(y;\Delta t)
&=\tfrac{1}{(2\pi)^2}\int_{\Rbb^2} (-i\eta)^\beta e^{-i\eta\cdot y} G(\eta;\Delta t)\,d\eta,
\\
\|\partial^\beta g(\cdot;\Delta t)\|_\infty
&\le \tfrac{1}{(2\pi)^2}\int_{\Rbb^2}\|\eta\|_2^{|\beta|}\,|G(\eta;\Delta t)|\,d\eta <\infty.
\end{align*}
For $(s_k,b_j)\in\Omega_{\myin}$, we write $x_{k,j}=(s_k,b_j)$, and for nodes on $\Omega$, we write $y_{l,d}=(s_l,b_d)$.

Define the (bounded) extension rectangle
\begin{equation}
\label{eq:omega_ext}
\Omega_{\mathrm{ext}}
:=
\big[s_{\min}^{\dagger}-s_{\max},\; s_{\max}^{\dagger}-s_{\min}\big]
\times
\big[b_{\min}^{\dagger}-b_{\max},\; b_{\max}^{\dagger}-b_{\min}\big],
\end{equation}
so that $\Omega-\{x_{k,j}\}\subseteq\Omega_{\mathrm{ext}}$ $\forall x_{k,j}\in\Omega_{\myin}$.
Then for any $(s_k,b_j)\in\Omega_{\myin}$ and any $\theta\in\widehat\Theta$,
\begin{align}
\Delta s\,\Delta b \sum_{l,d}
\varphi_{l,d}\,\big|\widehat g^{(\theta)}_{\,l-k,\,d-j}\big|
&=
\Delta s\,\Delta b \!\!\!\sum_{(y_{l,d}\in\Omega)}
\varphi_{l,d}\,\Big|\widehat g^{(\theta)}\!\big(y_{l,d}-x_{k,j}\big)\Big|
\nonumber\\[2pt]
&\le
\Delta s\,\Delta b \!\!\!\sum_{(y_{l,d}\in\Omega)}
\varphi_{l,d}\,\Big(g\big(y_{l,d}-x_{k,j}\big)
+ \big|\widehat g^{(\theta)} - g\big|\!\big(y_{l,d}-x_{k,j}\big)\Big)
\nonumber\\[2pt]
&\le
\int_{\Omega-\{x_{k,j}\}}\! g(z;\Delta t)\,dz
\;+\; C_q\,h^2
\;+\;|\Omega|\,\| \widehat g^{(\theta)} - g\|_{L_{\infty}(\Omega_{\mathrm{ext}})}
\nonumber\\[2pt]
&\le 1
\;+\; C_q\,h^2
\;+\;|\Omega|\,\| \widehat g^{(\theta)} - g\|_{L_{\infty}(\Omega_{\mathrm{ext}})}.
\label{eq:bound_g}
\end{align}
Here, $C_q>0$ is independent of $h$ and $(k,j)$ (composite trapezoid error
\mbox{on a uniform grid).}

Now, in \eqref{eq:bound_g}, set $\theta=\widehat\theta^\star$. By Corollary~\ref{cor:fournet_rates}
(under Assumption~\ref{ass:training_regime}), we have $\|\widehat g-g\|_{L_\infty(\Rbb^2)}\le C''h^{1+\kappa}$.
Hence, for sufficiently small $h$,
\begin{equation}
\label{eq:sum_ghat}
\Delta s\,\Delta b\!\!\mysum_{l,d}\big|\widehat g_{\,l-k,\,d-j}\big|
\le 1  +  C'\,h^{1+\kappa}  \!+\!  C_{q}\,h^2
\le  1 + \varepsilon(h) \le~ e^{\varepsilon(h)},\quad \varepsilon(h)=C\,h^{\,1+\kappa}.
\end{equation}

\subsection{Stability}
We now show $\ell_\infty$-stability for fixed $w_c$.

\begin{lemma}[$\ell_\infty$-stability for fixed $w_c$]
\label{lemma:stability}
Fix $w_c\in\Gamma_h$ and let $h>0$ be the global refinement parameter with
\eqref{eq:dis_parameter} and Assumption~\ref{ass:training_regime} satisfied.
If bilinear interpolation is used in \eqref{eq:num_intervention}, then the scheme
\eqref{eq:num_terminal}, \eqref{eq:num_bd_out}, and \eqref{eq:num_interior}--\eqref{eq:num_intervention}
is $\ell_\infty$-stable for each fixed $w_c$: there exist constants $h_0>0$ and $0<C<\infty$,
independent of $h$, such that for all $0<h\le h_0$ and all $m=0,\ldots,M$,
\[
\|V_c^{m,-}\|_\infty \;\le\; C,
\quad\text{ where } \quad
\|V_c^{m,-}\|_\infty := \max_{k\in\mathbb N_s^\dagger,\;j\in\mathbb N_b^\dagger}
|V_{k,j,c}^{m,-}|\;.
\]
\end{lemma}
A proof of Lemma~\ref{lemma:stability} is given in Appendix~\ref{app:stability}.

\subsection{Consistency}
We recall $s^+(x,q_m,u)$ and $b^+(x,q_m,u)$ from \eqref{eq:state-update}, and write
$x^+(x,q_m,u):=(s^+(x,q_m,u),\,b^+(x,q_m,u))$.  Denote
$\hat{x}=(x,w)$ and $\widehat{x}^{m,\circ}=(x,w,t_m^\circ)$ with $t_m^\circ\in\{t_m,t_m^\pm\}$.
We write the localized backward recursion at the reference point $\widehat{x}^{m,-}$ via the operator $\mathcal{D}(\cdot)$:
$V\big(\hat{x}^{m, -}\big) = \mathcal{D}\big(\hat{x}^{m, -},  V^{m+1, -}\big)= \ldots$
\begin{equation*}
\ldots
 =
 \begin{cases}
 \displaystyle \sup_{u\in\mathcal Z} \int_{\Omega}\! V(y,w,t_{m+1}^-)\;g\big(y-x^+(x,q_m,u);\Delta t\big)\,dy,
 & x \in \Omega_{\myin},
 ~~m = M-1, \ldots,0,
 \\
 G(-i\,a(x);\Delta t)\;V(x,w, t_{m+1}^-), & x \in \Omega_{\myout},~
 m = M-1, \ldots, 0,
 \\
\Phi(x, w),
& x \in \Omega,~~~~ m = M.
 \end{cases}
\end{equation*}
Let $\Omega^h\times\Gamma^h$ be the computational grid, with
$\Omega_{\myin}^h$ and $\Omega_{\myout}^h$ the interior/boundary sub‑grids.
The numerical scheme at the reference node $\widehat{x}^{m,-}_{k,j,c}=(s_k,b_j,w_c,t_m^-)$
is written via the discrete operator $\mathcal{D}_h(\cdot)$:
$V_{k, j, c}^{m, -} = \mathcal{D}_h\big(\hat{x}_{k, j, c}^{m, -}, \, \big\{ V_{l, d, c}^{m+1, -} \big\}\big)= \ldots$
\begin{equation*}
\ldots
 =
 \begin{cases}
\displaystyle \max_{u\in\mathcal Z_h}\;
\mathcal{I}\!\big[\{V_c^{m,+}\}\big]\!
\big(s_{k,j}^{m, +}(u),\,b_{k,j}^{m, +}(u)\big)
 & x_{k, j} \in \Omega_{\myin}^h,~~m = M-1, \ldots,0,
  \\
 G\big(-i\,a(s_k, b_j);\Delta t\big)\; V_{k,j, c}^{m+1,-}, & x_{k, j} \in \Omega_{\myout}^h,~
 m = M-1, \ldots, 0,
 \\
\Phi(s_k, b_j, w_c),
& x_{k, j} \in \Omega^h,~~~~ m = M.
 \end{cases}
\end{equation*}
where $ (s_{k,j}^{m, +}(u),\,b_{k,j}^{m, +}(u))
= (s^+(s_k,b_j, q_m, u), b^+(s_k,b_j, q_m, u))$ given as in \eqref{eq:state-update}.
\begin{lemma}[Pointwise consistency; fixed $w_c$]
\label{lemma:consistency}
Fix $w_c\in\Gamma_h$ and let $h>0$ satisfy \eqref{eq:dis_parameter} and
Assumption~\ref{ass:training_regime}. For any smooth test function
$\phi(\cdot,w_c,\cdot)\in C^\infty(\Omega\times[0,T])$, denote
$\phi^{m+1,-}(y):=\phi(y,w_c,t_{m+1}^-)$.
Then, for all $m=M-1,\ldots,0$ and all $x_{k,j}\in\Omega_{\myin}^h$,
\[
\mathcal{D}_h\big(\widehat{x}^{m,-}_{k,j,c},\,\{\phi(\widehat{x}_{l,d,c}^{m+1,-})\}_{l,d}\big)
-\mathcal{D}\big(\widehat{x}^{m,-}_{k,j,c},\,\phi^{m+1,-}\big)
= \mathcal{O}\big(h \,+\, h^{\,1+\kappa} \,+\, h^2\big).
\]
Moreover, $\mathcal{D}_h(\cdot)=\mathcal{D}(\cdot)$ on $\Omega^h$ for $m=M$, and on $\Omega_{\myout}^h$
for $m=M-1,\ldots,0$.
\end{lemma}
A proof of Lemma~\ref{lemma:consistency} is given in Appendix~\ref{app:consistency}.

\subsection{Monotonicity}
\begin{lemma}[Monotonicity]
\label{lemma:mon}
For fixed $w_c$ and any bounded data sets
$\{v_{l,d,c}^{m+1,-}\}_{l,d}$ and $\{z_{l,d,c}^{m+1,-}\}_{l,d}$ with
$\{v_{l,d,c}^{m+1,-}\}_{l,d} \le \{z_{l,d,c}^{m+1,-}\}_{l,d}$ (componentwise),
the discrete operator satisfies
\[
\mathcal{D}_h\big(\widehat{x}^{m,-}_{k,j,c},\, \{v_{l,d,c}^{m+1,-}\}_{l,d}\big)
\;\le\;
\mathcal{D}_h\big(\widehat{x}^{m,-}_{k,j,c},\, \{z_{l,d,c}^{m+1,-}\}_{l,d}\big),
\qquad \forall\, (k,j),\; m=M-1,\ldots,0.
\]
\end{lemma}

\begin{proof}
On $\Omega_{\myout}^h$, the update is multiplication by $G(-i\,a(\cdot);\Delta t)>0$, so order is preserved.
On $\Omega_{\myin}^h$, define the post-propagation arrays
\[
\tilde{v}_{k, j}
=\Delta s\,\Delta b \!\!\mysumsmall_{l,d}\!
\varphi_{l,d}\,\widehat g_{\,l-k,\,d-j}\,v_{l,d,c}^{m+1,-},
\qquad
\tilde{z}_{k, j}
=\Delta s\,\Delta b \!\!\mysumsmall_{l,d}\!
\varphi_{l,d}\,\widehat g_{\,l-k,\,d-j}\,z_{l,d,c}^{m+1,-}.
\]
Since $\varphi_{l,d}\ge0$ and $\widehat g_{\,l-k,\,d-j}\ge0$ (mixture weights $\beta_n\ge0$),
we have $\tilde v_{k,j}\le \tilde z_{k,j}$ componentwise on the full grid (including the boundary update).
Bilinear interpolation uses nonnegative weights summing to one, hence preserves order:
\[
\mathcal I\big[\{\tilde v\}\big]\big(x^{m,+}_{k,j}(u)\big)
\le
\mathcal I\big[\{\tilde z\}\big]\big(x^{m,+}_{k,j}(u)\big),
\qquad \forall\,u\in\mathcal Z_h.
\]
Taking $\max_{u\in\mathcal Z_h}$ preserves the inequality, yielding the claim.
\end{proof}

A generic grid point in $\Omega_{\myin}^h$ is denoted by $(s_h,b_h)$, and we write
$w_h\in\Gamma^h$ for a discrete threshold node. Let $V_h$ denote the numerical solution
produced by \eqref{eq:num_terminal}--\eqref{eq:num_interior} (with bilinear interpolation in
\eqref{eq:num_intervention}), and recall $V$ denotes the value function of the localized continuous problem.

\begin{lemma}[Convergence of the inner problem; fixed discrete threshold]
\label{lem:conv_inner}
Fix the global refinement parameter $h>0$ with \eqref{eq:dis_parameter} and
Assumption~\ref{ass:training_regime} satisfied, and fix any $w_h\in\Gamma^h$.
Let $(s',b')\in\Omega_{\myin}$ and let $\{(s_h,b_h)\}_{h\downarrow 0}$ be any sequence with
$(s_h,b_h)\in\Omega_{\myin}^h$ and $(s_h,b_h)\to(s',b')$ as $h\to 0$. Then, for each
$m\in\{M-1,\ldots,0\}$,
\begin{equation}
\label{eq:con_inner}
\bigl|\,V_h(s_h,b_h,w_h,t_m^-)\;-\;V(s',b',w_h,t_m^-)\,\bigr|\;\longrightarrow\;0
\qquad\text{as }h\to 0.
\end{equation}
\end{lemma}

\begin{proof}
Fix $w_h\in\Gamma^h$. By Lemma~\ref{lemma:stability}, Lemma~\ref{lemma:mon}, and
Lemma~\ref{lemma:consistency}, the scheme is $\ell_\infty$-stable, monotone, and pointwise consistent
(with error $\mathcal{O}(h)+\mathcal{O}(h^{1+\kappa})+\mathcal{O}(h^2)$). Thus the Barles--Souganidis
half-relaxed limits argument applies \cite{barles-souganidis:1991}. By the uniqueness of the localized inner value
function for fixed threshold (Remark~\ref{rm:uniqueness}), the upper and lower limits coincide with $V(\cdot,w_h,t_m^-)$, hence the desired result \eqref{eq:con_inner}.
\end{proof}

\begin{lemma}[Convergence of the outer optimization]
\label{lem:outer_conv}
Fix $x_0\in\Omega$ and, for each $h>0$, choose $x_0^h\in\Omega^h$ with $x_0^h\to x_0$.
Define the outer objectives at $t_0^-$ by
\[
F(w):=V(x_0,w,t_0^-),\qquad F_h(w_h):=V_h(x_0^h,w_h,t_0^-),
\]
for $w\in\Gamma$ and $w_h\in\Gamma_h$. Assume \eqref{eq:dis_parameter} and
Assumption~\ref{ass:training_regime}. Then:
\begin{enumerate}[noitemsep, topsep=2pt, leftmargin=*]
\item[\textup{(i)}] \emph{Convergence of optimal values.}
\[
\big|\sup_{w\in\Gamma} F(w)-\sup_{w_h\in\Gamma_h} F_h(w_h)\big|
\le
\mathcal{O}(h)+\mathcal{O}(h^{\,1+\kappa})+\mathcal{O}(h^2)
\;+\;L_w\,\Delta w
\;\xrightarrow[h\to 0]{}\;0,
\]
where $L_w:=\gamma\,(1+1/\alpha)$.

\item[\textup{(ii)}] \emph{Convergence of maximizers.}
Let $w_h^\star\in\arg\max_{w_h\in\Gamma_h}F_h(w_h)$. Every accumulation point
$\bar w$ of $\{w_h^\star\}$ belongs to $\arg\max_{w\in\Gamma}F(w)$. If the maximizer
of $F$ on $\Gamma$ is unique, then $w_h^\star\to w^\star$ as $h\to 0$.
\end{enumerate}
\end{lemma}

\begin{proof}
\emph{(i)} Let $\varepsilon_h:=\sup_{w_h\in\Gamma_h}|F_h(w_h)-F(w_h)|$.
By Lemma~\ref{lem:conv_inner} (at $m=0$) and the fact that stability/consistency constants
do not depend on $w$ (the threshold enters only through the terminal objective functional, which is matched exactly),
$\varepsilon_h=\mathcal{O}(h)+\mathcal{O}(h^{1+\kappa})+\mathcal{O}(h^2)$.
Moreover, for $w_1,w_2\in\Gamma$,
\[
|F(w_1)-F(w_2)|\le \gamma\Big(1+\tfrac1\alpha\Big)\,|w_1-w_2|=:L_w\,|w_1-w_2|,
\]
so $F$ is $L_w$--Lipschitz on $\Gamma$. The standard grid-approximation argument then yields
the stated bound.

\smallskip
\emph{(ii)} Let $w_h^\star\in\arg\max_{w_h\in\Gamma_h}F_h(w_h)$ and let $w^\star\in\arg\max_{w\in\Gamma}F(w)$.
Along any subsequence with $w_h^\star\to\bar w\in\Gamma$, we have
$F(\bar w)\ \ge\ F(w_h^\star)-L_w\,|w_h^\star-\bar w|
\ \ge\ F_h(w_h^\star)-\varepsilon_h-L_w\,|w_h^\star-\bar w|$,
and since $F_h(w_h^\star)=\sup_{\Gamma_h}F_h$,
$F_h(w_h^\star)=\sup_{\Gamma_h}F_h
\ \ge\ \sup_{\Gamma}F-\varepsilon_h-L_w\,\Delta w$,
it follows that
\[
F(\bar w)\ \ge\ \sup_{\Gamma}F\;-\;2\varepsilon_h\;-\;L_w\,\Delta w\;-\;L_w\,|w_h^\star-\bar w|.
\]
Letting $h\to 0$ (so that $\varepsilon_h\to 0$, $\Delta w\to 0$, and $|w_h^\star-\bar w|\to 0$) gives $F(\bar w)\ge \sup_{\Gamma}F$.
Since trivially $F(\bar w)\le \sup_{\Gamma}F$, we conclude $F(\bar w)=\sup_{\Gamma}F$, hence $\bar w\in\arg\max_{\Gamma}F$.
If the maximizer of $F$ on $\Gamma$ is unique, then the entire sequence $w_h^\star$ converges to $w^\star$.
\end{proof}
We now state the main convergence result in the next theorem.
\begin{theorem}[Main convergence of the full scheme]
\label{thm:main_convergence}
Fix $x_0\in\Omega$ and choose $x_0^h\in\Omega^h$ with $x_0^h\to x_0$ as $h\to0$; under \eqref{eq:dis_parameter} and Assumption~\ref{ass:training_regime}, the numerical outer-optimized value satisfies
\[
V_h(x_0^h,t_0^-):=\max_{w_h\in\Gamma_h}V_h(x_0^h,w_h,t_0^-)\to \max_{w\in\Gamma}V(x_0,w,t_0^-)=:V(x_0,t_0^-),
\]
and any $w_h^\star\in\arg\max_{\Gamma_h}V_h(x_0^h,\cdot,t_0^-)$ has accumulation points in $\arg\max_{\Gamma}V(x_0,\cdot,t_0^-)$ (with $w_h^\star\to w^\star$ if the maximizer is unique).
\end{theorem}
This follows from Lemmas~\ref{lem:conv_inner}–\ref{lem:outer_conv}.

\section{Generality and extensions}
\label{sec:generality_extensions}
Our monotone integration scheme is built from three structural ingredients: (i) translation-invariant inter-decision increments specified via a closed-form CF, (ii) a Fourier-trained transition density $\widehat g\ge 0$, and (iii) nonnegative quadrature/interpolation weights. Thus, for each fixed \mbox{auxiliary parameter $w$,} the inner Bellman recursion retains the same monotonicity/stability/consistency properties established in Lemmas~\ref{lemma:stability}--\ref{lem:conv_inner}. In particular, for each fixed $w$, the analysis up to Lemma~\ref{lem:conv_inner} is agnostic to the specific choice of terminal functional $\Phi(\cdot,w)$, which enters only through the terminal condition and the boundary conditions.

Accordingly, our results apply to optimization problems of the form
\[
\sup_{w\in\mathcal W}\ \sup_{\mathcal U_0\in\mathcal A}
\Ebb_{\mathcal U_0}^{x_0,t_0^-}\!\big[\Phi(W_T,w)\big].
\]
The outer-maximization convergence argument {extends whenever $\mathcal W$ is compact (or can be localized to one) and $w\mapsto\Phi(\cdot,w)$ is uniformly continuous on $\mathcal W$: there exists a modulus of continuity $\omega$ with $\omega(r)\downarrow 0$ as $r\downarrow 0$
such that
\[
\sup_{x\in\Omega}\,|\Phi(W(x),w_1)-\Phi(W(x),w_2)|
\ \le\ \omega(|w_1-w_2|),
\qquad \forall\,w_1,w_2\in\mathcal W.
\]
In this case, the outer grid-approximation error is controlled by $\omega(\Delta w)$ (cf.\ Lemma~\ref{lem:outer_conv}).

In many applications, $\Phi(W_T,w)$ admits a natural reward--risk decomposition that highlights a broad class of tail and shortfall criteria encompassed by the same $\sup_w\Ebb[\Phi(\cdot,w)]$ structure. Specifically, consider terminal functionals of the form
\[
\Phi(W_T,w)\;=\;R(W_T)\;+\;\gamma\,\varphi(Z,w),
\qquad
Z=\zeta(W_T).
\]
Here, $R(W_T)$ is a reward term (e.g.\ $R(W_T)=W_T$), $\gamma>0$ is a scalarization parameter, and $Z$ is a real-valued variable obtained from $W_T$ via a prescribed measurable mapping $\zeta:\Rbb_+\to\Rbb$ (e.g.\ a loss/shortfall metric derived from the aggregate level $W_T$).
The risk component is encoded through the integrand $\varphi(Z,w)$, which in many cases admits an auxiliary-variable representation of the form
$\rho(Z)\;=\;\sup_{w\in\mathcal W}\ \Ebb\!\big[\varphi(Z,w)\big]$,
so that a terminal reward--risk criterion of the form
$\sup_{\mathcal U_0\in\mathcal A}\big\{\Ebb[R(W_T)]+\gamma\,\rho(Z)\big\}$
can be written equivalently as
$\sup_{w\in\mathcal W}\ \sup_{\mathcal U_0\in\mathcal A}\,
\Ebb_{\mathcal U_0}^{x_0,t_0^-}\!\big[\,R(W_T)\;+\;\gamma\,\varphi(Z,w)\big]$,
which matches the $\sup_w\Ebb[\Phi(\cdot,w)]$ structure treated here (with $\Phi(W_T,w)=R(W_T)+\gamma\,\varphi(Z,w)$).

Beyond CVaR/expected shortfall \cite{RT2000}, this class also includes buffered probability of exceedance \cite{bpoe2018,dang2026multi} and optimized certainty equivalents / shortfall-type convex risk measures
\cite{BenTalTeboulle2007}, which can be cast into the same auxiliary-variable $\sup_w\Ebb[\varphi(\cdot,w)]$ representation (possibly after taking a negative when
originally posed via an infimum over $w$).
More generally, piecewise-linear (kinked) terminal penalties in $Z$ (and hence in $W_T$ via $Z=\zeta(W_T)$) with breakpoints parameterized by $w$ also fit this setting \cite{BenTalTeboulle2007}.

\paragraph{Higher dimensions.}
Finally, the Fourier-trained nonnegative kernel construction, the nonnegative-weight quadrature discretization, and the Barles--Souganidis convergence framework extend to dimensions $d>2$ under the same structural assumptions, with computational scaling corresponding to $d$-dimensional quadrature/FFT-based convolutions.

\section{Numerical experiments}
\label{sec:num_experiments}
This section reports numerical experiments addressing two separate questions:
(i) the approximation accuracy of the Fourier-trained transition kernel $\widehat g$ when the one-step increment law is specified only through a closed-form CF; and
(ii) the performance of the resulting monotone 2D integration scheme when $\widehat g$ is used to {\myblue{solve}} a representative multi-period \mbox{mean--CVaR optimization problem.}
Subsection~\ref{sec:2d_kou} addresses (i) using a fully coupled 2D jump--diffusion test law with synthetic parameters.
Subsection~\ref{sec:DC} then {\myblue{presents}} a Defined Contribution (DC) mean--CVaR portfolio illustration for working years (accumulation), where the same increment law is calibrated to long-horizon market data.

\paragraph{Error metrics.}
Let $f_1,f_2\in L_p(\Rbb^2)$, $p\in\{1,2\}$. We report the truncated $L_p$ error
$L_p(f_1,f_2)=\int_{[-A,A]^2}\big|f_1(x)-f_2(x)\big|^p\,dx$,
for $A>0$ chosen sufficiently large so that truncation has negligible effect at the displayed
precision. We also report the maximum pointwise error
$\mathrm{MPE}(f_1,f_2)=\max_{1\le k\le K}|f_1(x_k)-f_2(x_k)|$ over a fixed set of evaluation
points $\{x_k\}_{k=1}^K$. Among these, the $L_2$ error is the primary metric, consistent with the $L_2$ analysis in Section~\ref{sec:fournet}.

\subsection{Kernel learning accuracy}
\label{sec:2d_kou}
\subsubsection{2D Kou jump--diffusion}
We take $G(\eta)$ to be the one-step CF of a fully coupled 2D Kou jump--diffusion model, written in the L\'{e}vy-exponent form $G(\eta)=\exp(\Delta t\,\Psi(\eta))$ \cite{ken1999levy,kou01}. The characteristic exponent $\Psi(\eta_s,\eta_b)$ is given by
\begin{align}
\Psi(\eta_s,\eta_b)
&= i\,(\eta_s \mu_s^*+\eta_b \mu_b^*)
   - \tfrac12\big(\sigma_s^2\eta_s^2+\sigma_b^2\eta_b^2+2\rho\sigma_s\sigma_b\,\eta_s\eta_b\big)
\notag\\[-0.5mm]
&\quad +\,\lambda^s\!\big(\varphi_s(\eta_s)-1\big)
        + \lambda^b\!\big(\varphi_b(\eta_b)-1\big)
        + \lambda^c\!\big(\varphi_{c}(\eta_s,\eta_b)-1\big),
\qquad |\rho|<1,
\label{eq:ce_kou}
\end{align}
with idiosyncratic jump CFs
\[
\varphi_s(\eta_s)=p_s\,\frac{\eta_1^s}{\eta_1^s-i\eta_s}
+(1-p_s)\,\frac{\eta_2^s}{\eta_2^s+i\eta_s},\qquad
\varphi_b(\eta_b)=p_b\,\frac{\eta_1^b}{\eta_1^b-i\eta_b}
+(1-p_b)\,\frac{\eta_2^b}{\eta_2^b+i\eta_b},
\]
and a common (co-)jump component
\[
\varphi_c(\eta_s,\eta_b)=\varphi_{c,s}(\eta_s)\,\varphi_{c,b}(\eta_b),\qquad
\begin{cases}
\displaystyle \varphi_{c,s}(\eta_s)=p_{c,s}\,\frac{\eta_{1,c,s}}{\eta_{1,c,s}-i\eta_s}
+(1-p_{c,s})\,\frac{\eta_{2,c,s}}{\eta_{2,c,s}+i\eta_s},\\[2.0ex]
\displaystyle \varphi_{c,b}(\eta_b)=p_{c,b}\,\frac{\eta_{1,c,b}}{\eta_{1,c,b}-i\eta_b}
+(1-p_{c,b})\,\frac{\eta_{2,c,b}}{\eta_{2,c,b}+i\eta_b}.
\end{cases}
\]
The compensated drifts are
\[
\mu_s^* = \mu^s - \tfrac{(\sigma^s)^2}{2} - \lambda^s \kappa^s - \lambda^c \kappa^c_s, \quad
\mu_b^* = \mu^b - \tfrac{(\sigma^b)^2}{2} - \lambda^b \kappa^b - \lambda^c \kappa^c_b,
\]
with compensators
\[
\kappa^z=\varphi_z(-i)-1,\quad z\in\{s,b\},\quad
\kappa^c_s=\varphi_c(-i,0)-1,\quad\kappa^c_b=\varphi_c(0,-i)-1.
\]
The synthetic parameters are chosen as follows:
$\mu^s=0.08$, $\sigma^s=0.03$, $\lambda^s=0.6$, $p^s_{\mathrm{up}}=0.4$, $\eta^s_1=6.5$, $\eta^s_2=6.5$, $\rho_{sb}=0.05$;
$\mu^b=0.02$, $\sigma^b=0.04$, $\lambda^b=0.8$, $p^b_{\mathrm{up}}=0.5$, $\eta^b_1=20.5$, $\eta^b_2=22.5$;
$\lambda^c=0.2$, $p^c_{\mathrm{up},s}=0.5$, $\eta^c_{1,s}=25$, $\eta^c_{2,s}=30$,
$p^c_{\mathrm{up},b}=0.6$, $\eta^c_{1,b}=20$, $\eta^c_{2,b}=35$.
For concreteness, we set $\Delta t=1$ in the training experiments.

\subsubsection{Training setup and results}
\label{ssc:train_results}
We follow the Fourier-domain sampling and truncation prescriptions in
Sections~\ref{subsec:fournet_loss_error}--\ref{subsec:fournet_error}.
We choose $\eta'$ so that the tail conditions in \eqref{eq:boundD} hold with tolerance $\varepsilon=10^{-6}$;
in all experiments we use a conservative value $\eta'=80$ per axis (increasing $\eta'$ further does not
change the reported results at the shown precision). Table~\ref{tab:hparams} summarizes the training
hyperparameters.

\begin{table}[!htbp]
\centering
\begin{tabular}{lcccccc}
\toprule
$N$ & $P$ & \# epochs$_1$ & \# epochs$_2$ & $\ell_1$ & $\ell_2$ & \texttt{batchsize} \\
\midrule
60 & $10^{6}$ & 20 & 100 & 0.04 & 0.00025 & 1024 \\
\bottomrule
\end{tabular}
\caption{Hyperparameters for Fourier-domain neural-network training.}
\label{tab:hparams}
\end{table}
We use a two-stage optimizer schedule as in \cite{du2025fourier}: an initial exploration phase (AMSGrad
\cite{tran2019convergence}) followed by a refinement phase (Adam \cite{kinga2015method}).
All runs were conducted in Python/TensorFlow; hardware details do not materially affect the reported accuracy metrics.
In this example, both $\mathrm{Re},G$ and $\mathrm{Im},G$ oscillate in the frequency variable $\eta$. These oscillations may be rapid over moderate intervals (creating many local minima and noisy gradients) or relatively smooth but spread over a wide domain (reducing sampling and training efficiency). To improve conditioning, we apply a simple affine rescaling of the CF targets as in Remark~5.1 of \cite{du2025fourier}; the learned kernel is then mapped back and reported in the original coordinates.

Figure~\ref{fig:2dkuo_3} shows representative slices of $G$ and $\widehat G$, together with a corresponding slice of $\widehat g$. Errors are reported in Table~\ref{tab:2dkou_errors}, where the principal $L_2$ metric is highlighted. In this illustrative run, \mbox{the fitted CF yields $L_2$ errors on the order of $10^{-7}$}\footnote{
We also tested bivariate Variance--Gamma, bivariate NIG, and 2D Merton specifications; the qualitative kernel-fit accuracy and control outcomes are similar and are omitted for brevity.}.

\begin{figure}[htbp]
  \centering
  \subfigure[$\mathrm{Re}_G$, $\mathrm{Re}_{\widehat{G}}$]{%
    \includegraphics[width=0.32\linewidth]{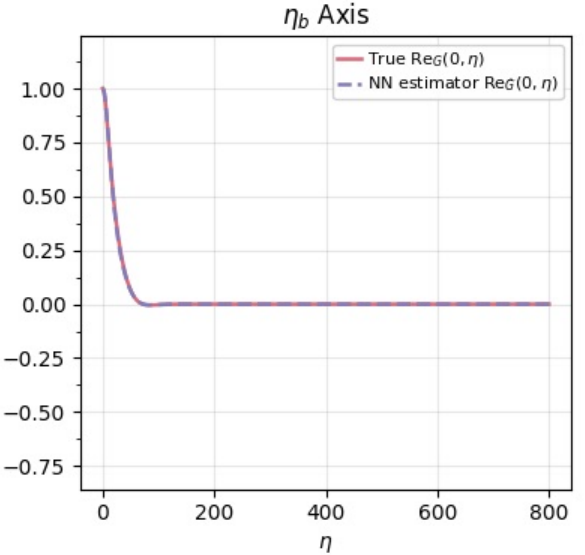}
    \label{fig:sub1}}%
  \hfill
  \subfigure[$\mathrm{Im}_G$, $\mathrm{Im}_{\widehat{G}}$]{%
    \includegraphics[width=0.32\linewidth]{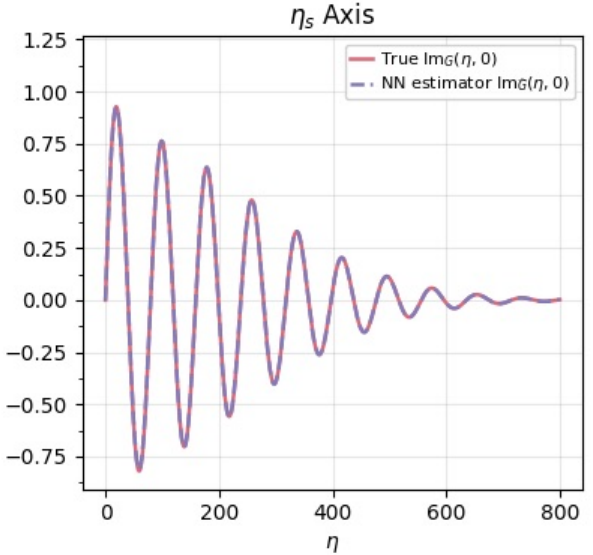}
    \label{fig:sub2}}%
  \hfill
  \subfigure[$\widehat{g}(\cdot;\Delta t)$]{%
    \includegraphics[width=0.32\linewidth]{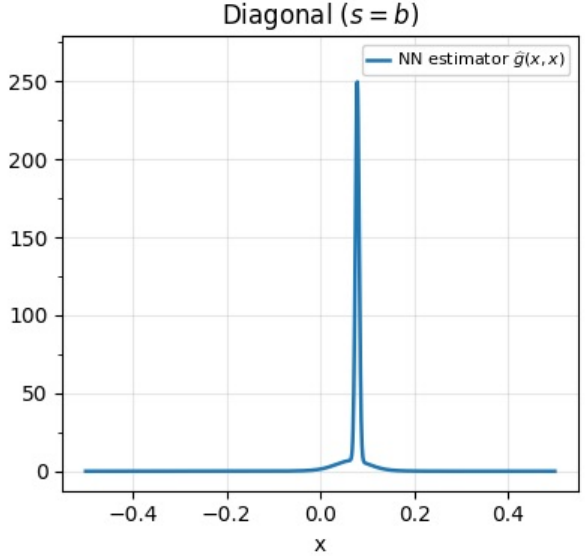}
    \label{fig:sub3}}
   \caption{2D Kou jump--diffusion test case. Panels (a)--(b) show representative 1D slices of the target CF $G$ and the fitted CF $\widehat G$, and panel (c) shows a representative slice of the trained density $\widehat g$. To improve conditioning for oscillatory $G$, the CF targets in the loss function are affinely rescaled as in Remark~5.1 of \cite{du2025fourier}; the plotted slices are shown in the original (unrescaled) coordinates.}
\label{fig:2dkuo_3}

\end{figure}

\begin{table}[htbp]
  \centering
  \renewcommand{\arraystretch}{1.3}
  \resizebox{\linewidth}{!}{%
  \begin{tabular}{c >{\columncolor{gray!20}}c c c >{\columncolor{gray!20}}c c}
    \toprule
     $L_1(\mathrm{Re}_G,\mathrm{Re}_{\widehat{G}})$ \,&
     $L_2(\mathrm{Re}_G,\mathrm{Re}_{\widehat{G}})$ \,&
     $\mathrm{MPE}(\mathrm{Re}_G,\mathrm{Re}_{\widehat{G}})$ \,&
     $L_1(\mathrm{Im}_G,\mathrm{Im}_{\widehat{G}})$ \,&
     $L_2(\mathrm{Im}_G,\mathrm{Im}_{\widehat{G}})$ \,&
     $\mathrm{MPE}(\mathrm{Im}_G,\mathrm{Im}_{\widehat{G}})$ \\
    \midrule
     1.1369e$-$05 \,& 9.2574e$-$07 \,& 3.6886e$-$04 \,&
     1.0682e$-$05 \,& 7.4543e$-$08 \,& 2.8006e$-$04 \\
    \bottomrule
  \end{tabular}
  }
  \caption{CF estimation errors for the 2D Kou jump--diffusion test case.}
  \label{tab:2dkou_errors}
\end{table}

\subsection{Mean–CVaR optimization (DC plan)}
\label{sec:DC}
We now specialize the generic 2-D mean--CVaR control problem to a DC accumulation illustration. At each intervention time, as is common in accumulation plans (see, e.g.\ \cite{dang2026multi, forsyth2020multiperiod}), the aggregate account value is allocated between the two assets with no borrowing and no negative positions, so the balance invested in each asset remains nonnegative. In this setting, the state $X_t=(S_t,B_t)$ represents the log-balances invested in an equity index (component $s$) and a bond index (component $b$), respectively.
The corresponding aggregate account value is $W_t=e^{S_t}+e^{B_t}$, so the terminal outcome $W_T$ in the mean--CVaR objective is the total retirement balance at horizon $T$.

Between decision times, the increment $\Delta X=((\Delta S)_m,(\Delta B)_m)$ represents the joint log-return over one period. At each intervention time $t_m$, the account receives a deterministic (salary) contribution $q_m$, and the control $u_m\in[0,1]$ reallocates the post-contribution wealth across the two components as in \eqref{eq:cash}--\eqref{eq:state-update}.

\subsubsection{Calibration summary}
To provide a concrete mean--CVaR control illustration with empirically realistic dynamics, we calibrate
the 2D Kou model of Subsection~\ref{sec:2d_kou} to long-horizon U.S.\ market data.
Specifically, we use monthly total return series from the Center for Research in Security Prices (CRSP)
over 1926:01--2024:12%
\footnote{The results presented here were calculated based on data from Historical Indexes,
\textcopyright{} 2015 Center for Research in Security Prices (CRSP),
The University of Chicago Booth School of Business.}.
In the DC illustration below, the two components correspond to {\myblue{equity}} and {\myblue{bond}} log-returns, respectively.
The resulting annualized parameter estimates are reported in Table~\ref{tab:double_exp_jump_b_c}. Dependence
is introduced via both the diffusion correlation and the common-jump component.

\begin{table}[htbp]
\centering
\resizebox{\linewidth}{!}{%
\begin{tabular}{lccccccc}
\toprule
Method & $\mu^s$ & $\sigma^s$ & $\lambda^s$ & $p^s_{\mathrm{up}}$ & $\eta^s_1$ & $\eta^s_2$ & $\rho_{sb}$ \\
\midrule
\multicolumn{8}{c}{{Component $s$ (Idiosyncratic; equity index)}} \\
 & 0.0898 & 0.1326 & 0.5960 & 0.373 & 6.701 & 6.634 & (see below) \\
\midrule
Method & $\mu^b$ & $\sigma^b$ & $\lambda^b$ & $p^b_{\mathrm{up}}$ & $\eta^b_1$ & $\eta^b_2$ & $\rho_{sb}$ \\
\midrule
\multicolumn{8}{c}{{Component $b$ (Idiosyncratic; bond index)}} \\
 & 0.0204 & 0.0466 & 0.9495 & 0.479 & 20.764 & 22.551 & 0.0721 \\
\midrule
\multicolumn{8}{c}{{Common Jump}} \\
\midrule
Method & $\lambda^c$ & $p^c_{\mathrm{up},s}$& $\eta^c_{1,s}$ & $\eta^c_{2,s}$ & $p^c_{\mathrm{up},b}$ & $\eta^c_{1,b}$ & $\eta^c_{2,b}$ \\
\midrule
& 0.1010 & 0.300 & 9.825 & 7.146 & 0.500 & 16.982 & 19.914 \\
\bottomrule
\end{tabular}
}
\caption{Estimated annualized parameters for the 2D Kou model  with idiosyncratic and common jumps, calibrated from CRSP equity index and 10-year Treasury total returns; monthly, 1926:01--2024:12. Threshold technique from \cite{dang2016better}.
}
\label{tab:double_exp_jump_b_c}
\end{table}

\subsubsection{Accumulation scenario}
\label{ssc:invest_scenario}
We now illustrate the full pipeline (kernel-learning + strictly monotone 2D integration) in a
multi-period mean--CVaR control problem with discrete interventions, using a defined DC retirement-accumulation setting. This DC example is used only as an application-level illustration (with interpretable interventions and
controls); the kernel-learning construction and convergence analysis are not specific to finance.

To illustrate the accumulation phase of a DC plan, we consider a 35-year-old investor with an annual
salary of \$100{,}000. The total contribution to the plan account is $20\%$ of salary each year.
The investor plans to retire at age 65, yielding a 30-year savings horizon \cite{forsyth2020multiperiod}.
The scenario is summarized in Table~\ref{tab:input-data}.
In this illustration we set $\Delta t=1$ year and use the calibrated parameters in
Table~\ref{tab:double_exp_jump_b_c}. We train $\widehat g$ using the same setup as in
Subsection~\ref{ssc:train_results}; the resulting CF-fit errors are of comparable magnitude to
Table~\ref{tab:2dkou_errors}, so we omit repeated slices/metrics for brevity.

\begin{table}[htbp]
  \centering
  \begin{tabular}{@{} l l @{}}
    \toprule
    Investment horizon (years) & 30 \\
    Initial investment $W_0$   & 0 \\
    Cash contributions         & \$20,000/year \\
    Rebalancing frequency      & yearly \\
    \bottomrule
  \end{tabular}
  \caption{DC accumulation illustration. Cash contributions are made at $t_m=0,1,\ldots,29$ years. Model parameters are given in Table~\ref{tab:double_exp_jump_b_c}.}
  \label{tab:input-data}
\end{table}

\subsubsection{Convergence results and efficient frontier}
We choose a computational domain large enough that boundary truncation effects are negligible
\cite{zhang2024monotone,DangForsyth2014}. For the calibrated parameters (Table~\ref{tab:double_exp_jump_b_c})
and the DC accumulation scenario (Table~\ref{tab:input-data}), the baseline domains are listed in
Table~\ref{tab:step01}. The padded bounds $(s_{\min}^\dagger,s_{\max}^\dagger)$ and
$(b_{\min}^\dagger,b_{\max}^\dagger)$ are set according to \eqref{eq:pad_bounds}.
Unless otherwise stated, refinement levels and discretization parameters are given in
Table~\ref{tab:disc_grid}.

\begin{table}[htbp]
\centering
\begin{minipage}[t]{0.4\linewidth}
\strut\vspace*{-\baselineskip}\newline
\flushleft
  \begin{tabular}{c|c}
    \hline
    $s_{\mymin}$ & $s_{\mymax}$ \\
    \hline
    $\ln(10^5)-8$ & $\ln(10^5)+8$ \\
    \hline  \hline
    $b_{\mymin}$ & $b_{\mymax}$ \\
    \hline
    $\ln(10^5)-8$ & $\ln(10^5)+8$ \\
    \hline  \hline
    \multicolumn{2}{c}{$w_{\mymax}$} \\
    \hline
    \multicolumn{2}{c}{$1\times 10^8$} \\
    \hline     \hline
  \end{tabular}
    \captionof{table}{Computational domains of numerical experiments.}
      \label{tab:step01}
\end{minipage}
\hfill
\begin{minipage}[t]{0.50\linewidth}
\strut\vspace*{-\baselineskip}\newline
\flushright

  \begin{tabular}{c c c c c}
    \toprule
    \multicolumn{1}{c}{Refine.\ level} &
    \multicolumn{1}{c}{$N_s$} &
    \multicolumn{1}{c}{$N_b$} &
    \multicolumn{1}{c}{$N_u$} &
    \multicolumn{1}{c}{$N_w$} \\
    \midrule
0 & 512  & 512  & 256 & 512 \\
    1 & 1024 & 1024 & 512 & 1024 \\
        2 & 2048 & 2048 & 1024 & 2048 \\
    \bottomrule
  \end{tabular}
    \captionof{table}{Discretization parameters and grid refinement levels.}
      \label{tab:disc_grid}
\end{minipage}
\end{table}
Table~\ref{tab:convergence_premcommitment} reports a convergence check for the
mean--CVaR optimization problem with $\gamma=10$ and $\alpha=0.05$.
The differences between the two finest refinement levels suggest that the computed objective is accurate to well within 1\%. The optimal controls are computed and stored, and then used as inputs to Monte Carlo simulation ($2.56\times 10^6$ simulations) for validation (reported in the same table).
\begin{table}[htpb]
\centering
\resizebox{\linewidth}{!}{%
\begin{tabular}{lccc|ccc}
\toprule
Refine. & \multicolumn{3}{c}{Our method}
& \multicolumn{3}{c}{Monte Carlo} \\
\cmidrule(lr){2-4} \cmidrule(lr){5-7}
level
& $\mathbb{E}[W_T]$ & CVaR (5\%) & $w^\ast$
& $\mathbb{E}[W_T]$ & CVaR (5\%) & Median[$W_T$] \\
\midrule
$0$
& 2765.81 & 624.08 & 733.86
& 2760.64 (6.1) & 624.50 & 1481.44 \\
$1$
& 2768.71 & 624.37 & 734.12
& 2768.12 (6.1) & 624.37 & 1481.34 \\
$2$
& 2769.90 & 624.42 & 734.25
& 2770.18 (6.1) & 624.41 & 1481.23 \\
\bottomrule
\end{tabular}%
}
\caption{Convergence test, mean--CVaR with $\gamma=10$ and $\alpha=0.05$. Parameters in Table~\ref{tab:double_exp_jump_b_c}.
Brackets show half-widths of 99\% confidence intervals.
Units: thousands of currency units.}
\label{tab:convergence_premcommitment}
\end{table}

\begin{figure}[htbp]
  \centering
  \begin{minipage}[t]{0.5\linewidth}
    \centering
    \includegraphics[width=0.9\linewidth]{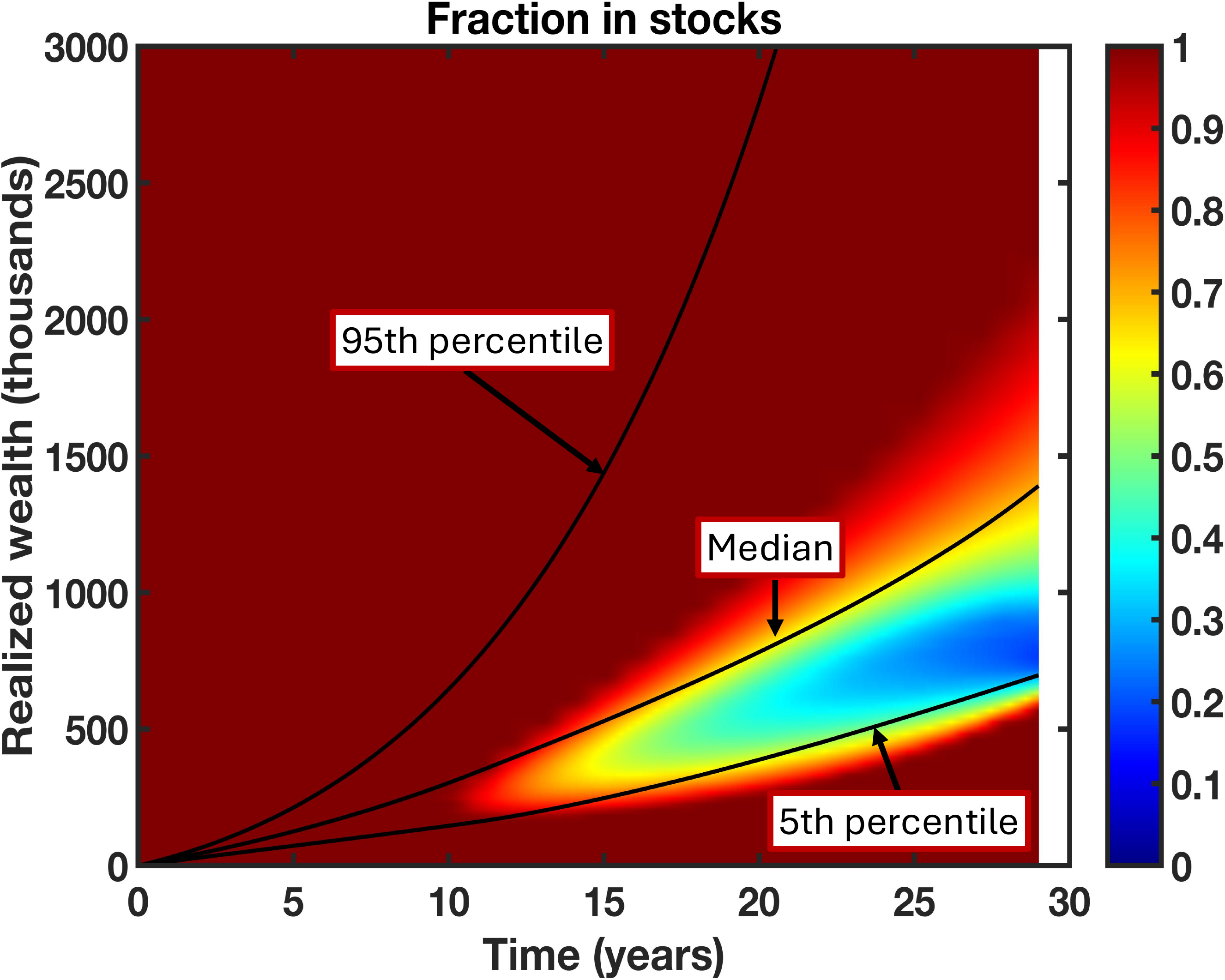}
    \caption{Pre-commitment mean--CVaR optimal control heat map (allocation weight in component $s$; interpreted as equity in the DC illustration).}
    \label{fig:hm_oc}
  \end{minipage}
  \hfill
  \begin{minipage}[t]{0.45\linewidth}
    \centering
    \includegraphics[width=0.9\linewidth]{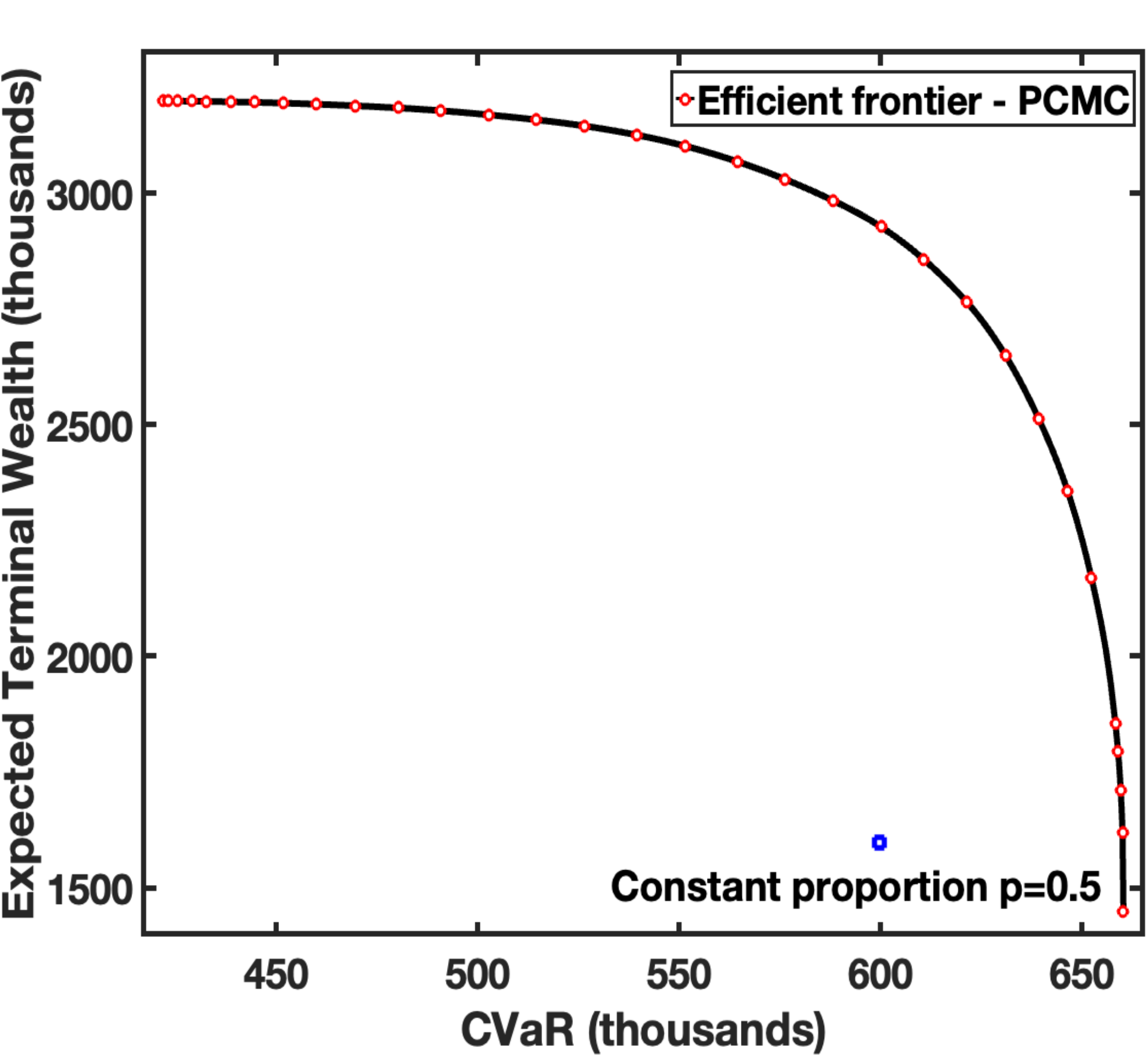}
    \caption{Efficient frontier of mean--CVaR with $\alpha=0.05$, computed on the finest refinement level (DC illustration).}
    \label{fig:ef}
  \end{minipage}
\end{figure}
Figure~\ref{fig:hm_oc} shows the optimal allocation weight in component $s$ as a function of time and
realized wealth. Early in the horizon the policy is near full allocation to component $s$; as time advances,
allocations become more state-dependent, with a clear de-risking region appearing around the wealth range
associated with the optimal threshold $w^\ast\approx 0.73$ million (Table~\ref{tab:convergence_premcommitment}).
For wealth well below this level, the policy increases exposure to component $s$ to improve tail outcomes; for
wealth near the threshold, the policy becomes more conservative; and for sufficiently high wealth, the CVaR
term is less binding and the allocation again tilts toward \mbox{higher expected growth.}

\medskip
\noindent To construct the efficient frontier, we vary the scalarization parameter $\gamma\in[0.01,1000]$ and repeat
the optimization at each $\gamma$ (all results on the finest refinement level). Figure~\ref{fig:ef} plots the
resulting frontier in the $(\mathrm{CVaR}_\alpha(W_T),\,\mathbb{E}[W_T])$ plane with \mbox{$\alpha=0.05$.} Since CVaR is applied to terminal wealth in this illustration, larger $\mathrm{CVaR}_\alpha(W_T)$ corresponds to better
downside performance. Increasing $\gamma$ shifts emphasis toward improving tail outcomes (higher CVaR) at the expense of expected terminal wealth.

\subsubsection{Robustness checks}
\label{ssc:spatial_size}
\paragraph{Impact of spatial domain sizes}
To validate that the baseline domains in Table~\ref{tab:step01} are sufficiently large, we repeat the experiment reported in Table~\ref{tab:convergence_premcommitment} on (i) a moderately larger domain and (ii) a smaller domain,
adjusting $N_s,N_b$ to keep $\Delta s,\Delta b$ unchanged. Tables~\ref{tab:impact_sds_l}--\ref{tab:impact_sds_s}
show that enlarging the domain has negligible impact at the reported precision, while shrinking the domain
introduces small discrepancies (as expected from truncation effects).

\begin{table}[htpb]
\centering
\begin{tabular}{lccc|ccc}
\toprule
& \multicolumn{3}{c}{Larger domain}
& \multicolumn{3}{c}{Table~\ref{tab:convergence_premcommitment}} \\
\cmidrule(lr){2-4} \cmidrule(lr){5-7}
Level
& $\mathbb{E}[W_T]$ & CVaR (5\%) & $w^\ast$
& $\mathbb{E}[W_T]$ & CVaR (5\%) & $w^\ast$ \\
\midrule
$0$
& 2765.80 & 624.08 & 733.86
& 2765.81 & 624.08 & 733.86 \\
$1$
& 2768.71 & 624.37 & 734.12
& 2768.71 & 624.37 & 734.12  \\
$2$
& 2769.90 & 624.42 & 734.25
& 2769.90 & 624.42 & 734.25  \\
\bottomrule
\end{tabular}%
\caption{Results using a larger spatial domain: $z_{\min}=\ln(10^5)-10$, $z_{\max}=\ln(10^5)+10$ (baseline domain $z_{\min}=\ln(10^5)-8$, $z_{\max}=\ln(10^5)+8$, Table~\ref{tab:step01}) for $z\in\{s,b\}$.}
\label{tab:impact_sds_l}
\end{table}

\begin{table}[htpb]
\centering
\begin{tabular}{lccc|ccc}
\toprule
& \multicolumn{3}{c}{Smaller domain}
& \multicolumn{3}{c}{Table~\ref{tab:convergence_premcommitment}} \\
\cmidrule(lr){2-4} \cmidrule(lr){5-7}
Level
& $\mathbb{E}[W_T]$ & CVaR (5\%) & $w^\ast$
& $\mathbb{E}[W_T]$ & CVaR (5\%) & $w^\ast$\\
\midrule
$0$
& 2765.01 & 624.10 & 733.86
& 2765.81 & 624.08 & 733.86 \\
$1$
& 2767.96 & 624.39 & 734.12
& 2768.71 & 624.37 & 734.12 \\
$2$
& 2769.19 & 624.45 & 734.25
& 2769.90 & 624.42 & 734.25 \\
\bottomrule
\end{tabular}%
\caption{Results using a smaller spatial domain: $z_{\min}=\ln(10^5)-6.25$, $z_{\max}=\ln(10^5)+6.25$ for $z\in\{s,b\}$; (baseline domain from Table~\ref{tab:step01}).}
\label{tab:impact_sds_s}
\end{table}

\paragraph{Impact of boundary conditions}
We repeat the experiment in Table~\ref{tab:convergence_premcommitment} using the constant boundary conditions proposed in \cite{zhou2025numerical,dang2025monotone}.  Table~\ref{tab:constant_bc} shows that, on the baseline domain, the constant boundary results are close to, but not fully identical with, those obtained using our asymptotic boundary conditions \eqref{eq:bd_operator}. As a robustness
check, Table~\ref{tab:constant_bc_l} repeats the constant boundary experiment on a larger domain, which reduces the discrepancy, consistent with constant boundary conditions requiring larger domains to reach comparable accuracy.
\begin{table}[htpb]
\centering
\begin{tabular}{lccc|ccc}
\toprule
& \multicolumn{3}{c}{Constant boundary}
& \multicolumn{3}{c}{Table~\ref{tab:convergence_premcommitment}} \\
\cmidrule(lr){2-4} \cmidrule(lr){5-7}
Level
& $\mathbb{E}[W_T]$ & CVaR (5\%) & $w^\ast$
& $\mathbb{E}[W_T]$ & CVaR (5\%) & $w^\ast$ \\
\midrule
$0$
& 2765.77 & 624.07 & 733.86
& 2765.81 & 624.08 & 733.86 \\
$1$
& 2768.64&  624.38 & 734.12
& 2768.71 & 624.37 & 734.12  \\
$2$
& 2769.76 & 624.42  & 734.25
& 2769.90 & 624.42 & 734.25  \\
\bottomrule
\end{tabular}%
\caption{Constant boundary conditions on the baseline domain, compared with our asymptotic boundary conditions \eqref{eq:bd_operator}.}
\label{tab:constant_bc}
\end{table}

\begin{table}[htpb]
\centering
\begin{tabular}{lccc|ccc}
\toprule
& \multicolumn{3}{c}{Constant boundary}
& \multicolumn{3}{c}{Tab.~\ref{tab:convergence_premcommitment}} \\
\cmidrule(lr){2-4} \cmidrule(lr){5-7}
Level
& $\mathbb{E}[W_T]$ & CVaR (5\%) & $W^\ast$
& $\mathbb{E}[W_T]$ & CVaR (5\%) & $W^\ast$ \\
\midrule
$0$
& 2765.80 & 624.08 & 733.86
& 2765.81 & 624.08 & 733.86 \\
$1$
& 2768.71 & 624.37 & 734.12
& 2768.71 & 624.37 & 734.12  \\
$2$
& 2769.90 & 624.42 & 734.25
& 2769.90 & 624.42 & 734.25  \\
\bottomrule
\end{tabular}%
\caption{Constant boundary conditions on a larger domain, compared with our asymptotic boundary conditions \eqref{eq:bd_operator}.}
\label{tab:constant_bc_l}
\end{table}

\section{Conclusion}
We developed a strictly monotone 2D integration scheme for multi-period mean--CVaR optimization, a representative class of reward--risk stochastic control problems with discrete interventions, in settings where the between-intervention increment law is specified via a closed-form CF with mild Fourier-tail decay.

The key computational ingredient is a Fourier-trained, nonnegative Gaussian-mixture transition kernel, which enables direct composite-quadrature evaluation of the Bellman convolution and an efficient FFT implementation. Our error analysis is conducted in Fourier space, leveraging the explicit Fourier form of both the Gaussian-mixture kernel and the target CF: we derive Fourier-domain $L_2$ error estimates and translate them into real-space bounds, which are then used to establish $\ell_\infty$ stability, consistency, and pointwise convergence as the discretization and kernel-approximation parameters vanish.

Numerical experiments include (i) a fully coupled 2-D jump--diffusion test law to assess kernel-learning accuracy, and (ii) a DC mean--CVaR optimization example representing accumulation over working years, calibrated to long-horizon data.
Together, they demonstrate practical accuracy and robustness. More broadly, the Fourier-trained kernel construction and the Fourier-to-real-space convergence framework apply to other reward--risk stochastic control problems with translation-invariant kernels, and can be extended beyond two dimensions, subject to the standard computational considerations.
Future work includes time-consistent optimization formulations and adaptive Fourier training strategies that preserve strict monotonicity.

\appendix
\section{Proof of Lemma~\ref{lem:CVaR_threshold_existence}}
\label{app:CVaR_threshold_existence}

For each $w\ge 0$, define $F(w):=\ds\sup_{\mathcal{U}_0 \in \mathcal{A}}
\Ebb_{\mathcal{U}_0}^{x_0,\,t_{0 }^{-}}
\big[ W_T + \gamma\big(w+\tfrac{1}{\alpha} \min \left(W_T-w, 0\right)\big)\big]$.
By Remark~\ref{rmk:finiteness_W},
$K:=\sup_{\mathcal{U}_0\in\mathcal A}\Ebb_{\mathcal{U}_0}^{x_0,t_0^-}[W_T]<\infty$, hence $F(w)$ is finite for each $w\ge 0$.
Next, use $\min(W_T-w,0)\le W_T-w$ to obtain, for any admissible $\mathcal U_0$,
\[
W_T+\gamma\Big(w+\tfrac1\alpha \min(W_T-w,0)\Big)
\;\le\;
\Big(1+\tfrac{\gamma}{\alpha}\Big)W_T \;+\;\gamma\Big(1-\tfrac1\alpha\Big)w.
\]
Taking expectations and then $\ds\sup_{\mathcal U_0\in\mathcal A}$ yields the uniform bound
\[
F(w)\;\le\;\Big(1+\tfrac{\gamma}{\alpha}\Big)K \;+\;\gamma\Big(1-\tfrac1\alpha\Big)w.
\]
Since $1-\tfrac1\alpha<0$, the right-hand side tends to $-\infty$ as $w\to\infty$, hence $F(w)\to -\infty$ as $w\to\infty$.
Finally, for each fixed $\mathcal U_0$, the map
$w\mapsto w+\tfrac1\alpha\min(W_T-w,0)$ is Lipschitz in $w$ (uniformly in $W_T$), so
$w\mapsto \Ebb_{\mathcal U_0}^{x_0,t_0^-}[\cdot]$ is continuous; therefore $F(\cdot)=\sup_{\mathcal U_0}(\cdot)$ is Lipschitz and in particular continuous.
Thus, there exists $\overline w<\infty$ such that $\sup_{w\ge 0}F(w)=\max_{w\in[0,\overline w]}F(w)$, and by continuity the maximum is attained at some $w^*(x_0)\in[0,\overline w]\subset[0,\infty)$.

\section{Proof of Lemma~\ref{lem:fournet_L2_bound}}
\label{app:fournet_L2_bound}
Let $\Delta(\eta):=G(\eta)-\widehat G(\eta;\widehat\theta^\star)$.

\noindent\emph{(i) $L_2$ bound.}
By \eqref{eq:Fourier_L2_isometry} applied to $f=g-\widehat g$,
\[
\int_{\Rbb^2}|g-\widehat g|^2
=\frac{1}{(2\pi)^2}\int_{\Rbb^2}|\Delta(\eta)|^2\,d\eta.
\]
Split $\Rbb^2=D_\eta\cup(\Rbb^2\setminus D_\eta)$. On $\Rbb^2\setminus D_\eta$,
$|\Delta|^2\le 2(|G|^2+|\widehat G|^2)$, so \eqref{eq:boundD} yields
$\int_{\Rbb^2\setminus D_\eta}|\Delta|^2<4\varepsilon_1$.
On $D_\eta$, approximate $\int_{D_\eta}|\Delta|^2$ by a composite left-hand quadrature
$\sum_{p=1}^P w_p |\Delta(\eta_p)|^2$. Under \eqref{eq:delta_eta},
the weights satisfy $0<w_p\le C_1/P$ and the quadrature remainder is bounded by
$\frac{C' C_1^3}{2P^{1/2}}$ (using \eqref{eq:grad}). Hence
\[
\int_{D_\eta}|\Delta|^2\,d\eta
\le C_1\,\frac{1}{P}\sum_{p=1}^P|\Delta(\eta_p)|^2 + \frac{C' C_1^3}{2P^{1/2}}
\le C_1\,\mathrm{Loss}_P(\widehat\theta^\star)+\frac{C' C_1^3}{2P^{1/2}},
\]
and since $\mathrm{Loss}_P(\widehat\theta^\star)<\varepsilon_2$, this gives \eqref{eq:fournet_L2_bound}.

\noindent\emph{(ii) Pointwise bound.}
From the inverse transform \eqref{eq:FT_pair} and $|a+ib|\le |a|+|b|$,
\[
|g(y)-\widehat g(y;\widehat\theta^\star)|
\le \frac{1}{(2\pi)^2}\int_{\Rbb^2}\big(|\mathrm{Re}_\Delta(\eta)|+|\mathrm{Im}_\Delta(\eta)|\big)\,d\eta.
\]
Split $\Rbb^2=D_\eta\cup(\Rbb^2\setminus D_\eta)$. On the tail,
\[
\int_{\Rbb^2\setminus D_\eta}\big(|\mathrm{Re}_\Delta|+|\mathrm{Im}_\Delta|\big)\,d\eta
\le \int_{\Rbb^2\setminus D_\eta}(|\mathrm{Re}_G|+|\mathrm{Im}_G|) \,d\eta
+\int_{\Rbb^2\setminus D_\eta}(|\mathrm{Re}_{\widehat G}|+|\mathrm{Im}_{\widehat G}|)
\le 2\varepsilon_1 \,d\eta.
\]
On $D_\eta$, applying the same quadrature bound with weights $w_p\le C_1/P$ gives
\[
\int_{D_\eta}\big(|\mathrm{Re}_\Delta|+|\mathrm{Im}_\Delta|\big)\,d\eta
\le \sum_{p=1}^P w_p\big(|\mathrm{Re}_\Delta(\eta_p)|+|\mathrm{Im}_\Delta(\eta_p)|\big)
+\frac{C' C_1^3}{2P^{1/2}}
\le C_1 R_P(\widehat\theta^\star)+\frac{C' C_1^3}{2P^{1/2}}.
\]
Since $R_P(\widehat\theta^\star)<\varepsilon_3$, \eqref{eq:Linfty_and_nonneg} follows.

\section{Proof of Lemma~\ref{lem:fourier_truncation_tails}}
\label{app:fourier_tails}
Let $\lambda:=c_0\,\Delta t>0$ and fix $\eta'\ge R$. For any $\eta\in\Rbb^2$ with
$\|\eta\|_\infty>\eta'$ we have $|\eta|>\eta'$. Under Assumption~\ref{ass:modelling}\,(A4),
\[
|G(\eta;\Delta t)|^p
\;\le\; \exp\!\big(-p\,\lambda\,|\eta|^\alpha\big),
\qquad p\in\{1,2\}.
\]
Hence, for $p\in\{1,2\}$, noting $\Rbb^2\setminus[-\eta',\eta']^2 = \{\eta \in \Rbb^2 | \|\eta\|_\infty>\eta'\}$,
\[
\int_{\{\|\eta\|_\infty>\eta'\}} |G(\eta;\Delta t)|^p\,d\eta
\;\le\; \int_{\{|\eta|>\eta'\}} e^{-p\lambda |\eta|^\alpha}\,d\eta
\;=\; 2\pi \int_{\eta'}^\infty r\,e^{-p\lambda r^\alpha}\,dr.
\]
Make the substitution $u=p\lambda r^\alpha$; then $r=(u/(p\lambda))^{1/\alpha}$ and
$dr=\frac{1}{\alpha}(p\lambda)^{-1/\alpha}u^{\frac{1}{\alpha}-1}\,du$,
\[
2\pi \int_{\eta'}^\infty r\,e^{-p\lambda r^\alpha}\,dr
\;=\; \frac{2\pi}{\alpha}(p\lambda)^{-2/\alpha}
\;\Gamma\!\Big(\frac{2}{\alpha},\,p\lambda\,\eta'^{\alpha}\Big),
\]
where $\Gamma(a,z)$ is the upper incomplete gamma function.
It is standard that for any $a>0$ and any $\theta\in(0,1)$ there exists
$C(a,\theta)>0$ such that
$\Gamma(a,z)\le C(a,\theta)\,e^{-\theta z}$ for all $z\ge 0$ \cite{AbramowitzStegun1972}.
Taking $\theta=\tfrac12$ and $z=p\lambda\,\eta'^\alpha$ in this bound, noting $\lambda = c_0\,\Delta t$, yields
\[
\int_{\{\|\eta\|_\infty>\eta'\}} |G(\eta;\Delta t)|^p\,d\eta
\;\le\;
C_G^{(p)} \exp\!\big(-c_G^{(p)}\,\eta'^\alpha\big),
\]
for constants
$ C_G^{(p)} = C(a,\theta)\,\frac{2\pi}{\alpha}\,\big(p\,c_0\,\Delta t\big)^{-2/\alpha}$ and $c_G^{(p)} \;:=\; \tfrac12\,p\,c_0\,\Delta t > 0$, independent of $\eta'\ge R$. This concludes
the Fourier domain truncation error bound related to~$G$.

\smallskip
\noindent
Regarding $\widehat G$, for $\eta=(\eta_s,\eta_b)\in\mathbb{R}^2$,
set $r:=\sigma^s_n\eta_s$ and $q:=\sigma^b_n\eta_b$. We then have
\[
\eta^\top\Sigma_n\eta
= r^2 + q^2 + 2\rho_n rq
\;\ge\; r^2+q^2 - 2|\rho_n||rq|
\;\ge\; (1-|\rho_n|)\,(r^2+q^2)
\;\ge\; (1-\overline\rho)\,\sigma_{\min}^2\,|\eta|^2.
\]
Hence,
$
|\widehat G(\eta;\theta)|
\;\le\;\sum_{n=1}^N|\beta_n|\,e^{-\frac12\,\eta^\top\Sigma_n\eta}
\;\le\; N\overline\beta\;\exp\!\Big(-\frac{(1-\overline\rho)\sigma_{\min}^2}{2}\,|\eta|^2\Big)$.

Let $\sigma^2:= (1-\overline\rho)\sigma_{\min}^2$. For $p\in\{1,2\}$,
\[
\int_{\{\|\eta\|_\infty>\eta'\}} |\widehat G(\eta;\theta)|^p\,d\eta
\;\le\; (N\overline\beta)^p \int_{\{|\eta|>\eta'\}} e^{-\frac{p\sigma^2}{2}\,|\eta|^2}\,d\eta
\;=\; (N\overline\beta)^p\,\frac{2\pi}{p\sigma^2}\,e^{-\frac{p\sigma^2}{2}\,\eta'^2}.
\]
Choosing $c_{\widehat G}^{(p)}\;:=\;\frac{p\;\sigma^2}{2}\;=\;\frac{p\;(1-\overline\rho)\sigma_{\min}^2}{2}$
and $C_{\widehat G}^{(p)}\;:=\;(N\overline\beta)^p\,\frac{2\pi}{p\sigma^2}$
concludes the proof.

\section{Proof of Corollary~\ref{cor:fournet_rates}}
\label{app:fournet_rates}
Choose $\eta'(h)$ so that the tail terms in Lemma~\ref{lem:fournet_L2_bound} satisfy
$\varepsilon_1\le c_1 h^{1+\kappa}$ (Assumption~\ref{ass:training_regime}).
Assume $\varepsilon_2\le c_2 h^{1+\kappa}$, $\varepsilon_3\le c_3 h^{1+\kappa}$, and
$P\ge c_P h^{-2(1+\kappa)}$, so $P^{-1/2}\le c_P^{-1/2}h^{1+\kappa}$.
Applying Lemma~\ref{lem:fournet_L2_bound}~(i) gives
$\|g-\widehat g\|_{L_2}^2\le C h^{1+\kappa}$, hence $\|g-\widehat g\|_{L_2}\le C h^{(1+\kappa)/2}$.
Applying Lemma~\ref{lem:fournet_L2_bound}~(ii) gives \eqref{eq:pointwise_g}.

\section{Proof of Lemma~\ref{lemma:stability}}
\label{app:stability}
By \eqref{eq:sum_ghat}, there exists $h_0>0$ such that for all $0<h\le h_0$,
\begin{equation}
\label{eq:mass_bound}
\Delta s\,\Delta b\mysum_{l,d}\varphi_{l,d}\big|\widehat g_{\,l-k,\,d-j}\big|
\;\le\;1+\varepsilon(h) \le e^{\varepsilon(h)},
\quad\text{for all } (s_k,b_j)\in\Omega_{\myin},
\end{equation}
with $\varepsilon(h)=C_\varepsilon h^{1+\kappa}$ and $C_\varepsilon$ independent of $h$. Define
\[
\kappa_{\max}
:=\frac{1}{\Delta t}\max\left\{\ln|G(-i,0;\Delta t)|,\ \ln|G(0,-i;\Delta t)|,\
\ln|G(-i,-i;\Delta t)|,\ 0\right\}.
\]
We claim that, for $m=0,\ldots,M$,
\begin{equation}
\label{eq:key}
\|V_c^{m,-}\|_\infty \le e^{(M-m)\,(\Delta t\,\kappa_{\max}+ \varepsilon(h))}\,\|V_c^{M,-}\|_\infty.
\end{equation}
The base case $m=M$ is immediate.

Assume \eqref{eq:key} holds at $m+1$.
For boundary nodes $(s_k,b_j)\in\Omega_{\myout}$, \eqref{eq:num_bd_out} and \eqref{eq:ax} give
\[
|V_{k,j,c}^{m,-}|=|V_{k,j,c}^{m,+}|
=|G(-i\,a(s_k,b_j);\Delta t)|\,|V_{k,j,c}^{m+1,-}|
\le e^{\Delta t\,\kappa_{\max}}\,\|V_c^{m+1,-}\|_\infty .
\]
For interior nodes $(s_k,b_j)\in\Omega_{\myin}$, \eqref{eq:num_interior} and \eqref{eq:mass_bound} yield
\[
|V_{k,j,c}^{m,+}|
\le
\Big(\Delta s\,\Delta b\mysum_{l,d}\varphi_{l,d}|\widehat g_{\,l-k,\,d-j}|\Big)\,\|V_c^{m+1,-}\|_\infty
\le e^{\varepsilon(h)}\,\|V_c^{m+1,-}\|_\infty .
\]
The intervention step \eqref{eq:num_intervention} uses bilinear interpolation with nonnegative
weights summing to one, followed by a max over controls, hence it is non-expansive in $\ell_\infty$:
$\|V_c^{m,-}\|_\infty\le \|V_c^{m,+}\|_\infty$. Therefore
\[
\|V_c^{m,-}\|_\infty \le e^{\Delta t\,\kappa_{\max}+\varepsilon(h)}\,\|V_c^{m+1,-}\|_\infty,
\]
which, combined with the induction hypothesis, implies \eqref{eq:key}.

Next,
\begin{equation}
\label{eq:bound_1}
\|V_c^{M,-}\|_\infty
\le
\sup_{x\in\Omega,\;w\in\Gamma}|V(x,w,T^-)| <\infty,
\end{equation}
since $\Omega\times\Gamma$ is compact and $V(\cdot,\cdot,T^-)$ from \eqref{eq:terminal} is continuous.
Choose $h_0$ smaller if needed so that $M\,\varepsilon(h)\le 1$ for $0<h\le h_0$. Then for all $m$,
\begin{equation}
\label{eq:bound_2}
e^{(M-m)(\Delta t\,\kappa_{\max}+ \varepsilon(h))}
\le e^{T\,\kappa_{\max}+1}.
\end{equation}
Combining \eqref{eq:key}--\eqref{eq:bound_2} yields the stated uniform $\ell_\infty$ bound.

\section{Proof of Lemma~\ref{lemma:consistency}}
\label{app:consistency}
\begin{proof}
The equality $\mathcal{D}_h=\mathcal{D}$ is immediate at $m=M$ (terminal payoff) and on $\Omega_{\myout}^h$ for $m = M-1, \ldots, 0$. We therefore focus on
the interior.

Fix $(s_k,b_j)\in\Omega_{\myin}^h$ and write $x_{k,j}=(s_k,b_j)$.
Let $\phi_{m+1}(y):=\phi(y,w_c,t_{m+1}^-)$ and define the continuous propagator
\[
\mathcal{Q}(x):=
{\myblue{\mathbf{1}_{\Omega_{\myin}}(x)\int_{\Omega}\phi_{m+1}(y)\,g(y-x;\Delta t)\,dy
+\mathbf{1}_{\Omega_{\myout}}(x)\,G(-i\,a(x);\Delta t)\,\phi_{m+1}(x)}},
\]
\[
\widetilde{\mathcal{Q}}(x):=
{\myblue{\mathbf{1}_{\Omega_{\myin}}(x)\int_{\Omega}\phi_{m+1}(y)\,\widehat g(y-x;\widehat\theta^\star)\,dy
+\mathbf{1}_{\Omega_{\myout}}(x)\,G(-i\,a(x);\Delta t)\,\phi_{m+1}(x)}}.
\]
Define the discrete propagation values (using the same $\widehat g$ on $\Omega_{\myin}^h$
and the boundary update on $\Omega_{\myout}^h$)
\begin{align*}
\mathcal{Q}_h(x_{l',d'}) &:=
{\myblue{\mathbf{1}_{\Omega_{\myin}^h}(x_{l',d'})}}\,
\Delta s\,\Delta b\sum_{l,d}\varphi_{l,d}\,
\phi(\widehat{x}^{m+1,-}_{l,d,c})\,
\widehat g\big(y_{l,d}-x_{l',d'};\widehat\theta^\star\big)
\\
& \qquad \;+\;
{\myblue{\mathbf{1}_{\Omega_{\myout}^h}(x_{l',d'})\,G(-i\,a(x_{l',d'});\Delta t)\,
\phi(\widehat{x}^{m+1,-}_{l',d',c})}},
\qquad x_{l',d'}\in\Omega^h,
\end{align*}
and let $\widetilde{\mathcal{Q}}_h(x_{l',d'}):=\widetilde{\mathcal{Q}}(x_{l',d'})$ denote nodal samples of $\widetilde{\mathcal{Q}}$ on $\Omega^h$.

With $x^{m,+}_{k,j}(u):=x^+(x_{k,j},q_m,u) \equiv (s_{k,j}^{m, +}(u),\,b_{k,j}^{m, +}(u))$, $\mathcal{D}(\cdot)$ and $\mathcal{D}_h(\cdot)$ \mbox{can be written as}
\begin{align*}
\mathcal{D}\big(\widehat{x}^{m,-}_{k,j,c},\,\phi^{m+1,-}\big)
&=\sup_{u\in\mathcal Z}\,\mathcal{Q}\big(x^{m,+}_{k,j}(u)\big),
\\
\mathcal{D}_h\big(\widehat{x}^{m,-}_{k,j,c},\,\{\phi(\widehat{x}^{m+1,-}_{l,d,c})\}_{l,d}\big)
&=\max_{u\in\mathcal Z_h}\;
\mathcal{I}\big[\big\{\mathcal{Q}_h(x_{l',d'})\big\}_{l', d'}\big]\!\big(x^{m,+}_{k,j}(u)\big).
\end{align*}
For any $u\in\mathcal Z_h$, set $\xi:=x^{m,+}_{k,j}(u)$. Using bilinear interpolation,
\begin{align}
\big|\mathcal{I}[\mathcal{Q}_h](\xi)\!-\!\mathcal{Q}(\xi)\big|
\le
\big|\mathcal{I}[\mathcal{Q}_h](\xi) \!-\! \mathcal{I}[\widetilde{\mathcal{Q}}_h](\xi)\big|
\!+\!\big|\mathcal{I}[\widetilde{\mathcal{Q}}_h](\xi)\!-\!\widetilde{\mathcal{Q}}(\xi)\big|
\!+\!\big|\widetilde{\mathcal{Q}}(\xi)\!-\!\mathcal{Q}(\xi)\big|.
\label{eq:cons_core_decomp}
\end{align}

\emph{Quadrature term.}
Since $y\mapsto \phi_{m+1}(y)\widehat g(y-x)$ is smooth on $\Omega$ with uniformly bounded
second derivatives (Gaussian mixture kernel; $\phi$ smooth), the 2-D composite trapezoid
rule gives, for each ${\myblue{x_{l',d'}\in\Omega_{\myin}^h}}$,
\[
\big|\mathcal{Q}_h(x_{l',d'})-\widetilde{\mathcal{Q}}_h(x_{l',d'})\big|\le C_q h^2,
\]
{\myblue{and for $x_{l',d'}\in\Omega_{\myout}^h$ the difference is $0$ by construction.}}
Hence
\[
\big\|\mathcal{Q}_h-\widetilde{\mathcal{Q}}_h\big\|_{L_\infty(\Omega)}\le C_q h^2,
\quad\Rightarrow\quad
\big|\mathcal{I}[\mathcal{Q}_h](\xi)- \mathcal{I}[\widetilde{\mathcal{Q}}_h](\xi)\big|\le C_q h^2.
\]

\emph{Interpolation term.}
{\myblue{On each region where its definition is smooth (in particular on $\Omega_{\myin}$ and on each boundary subdomain where $a(\cdot)$ is constant),}
$\widetilde{\mathcal{Q}}$ has bounded second derivatives}, hence bilinear
interpolation yields
\begin{equation}
\label{eq:cons_quad}
\big|\mathcal{I}[\widetilde{\mathcal{Q}}_h](\xi)-\widetilde{\mathcal{Q}}(\xi)\big|
\;\le\;C_{\mathrm{int}}\,h^2,
\end{equation}
uniformly in $\xi\in\Omega$, with  bounded constant $C_{\mathrm{int}}>0$ independent of $h$.

\emph{Kernel mismatch term.}
{\myblue{If $\xi\in\Omega_{\myout}$, then $\widetilde{\mathcal{Q}}(\xi)=\mathcal{Q}(\xi)$ by definition. If $\xi\in\Omega_{\myin}$, using the global pointwise estimate
$\|\widehat g-g\|_{L_\infty(\Rbb^2)}=O(h^{1+\kappa})$ (Assumption~\ref{ass:training_regime}), we have}}
$\big|\widetilde {\mathcal{Q}}(\xi)-{\mathcal{Q}}(\xi)\big| = \ldots$
\begin{equation}
\label{eq:cons_kernel}
\ldots=\Big|\!\!\int_{\Omega}\!\!\phi_{m+1}(y)\big(\widehat g-g\big)(y-\xi)\,dy\Big|
\!\le\! \|\phi_{m+1}\|_{L_\infty(\Omega)}\,|\Omega|\,\|\widehat g-g\|_{L_\infty(\Rbb^2)}
\!=\! C_g\,h^{1+\kappa}.
\end{equation}
Combining \eqref{eq:cons_core_decomp}--\eqref{eq:cons_kernel} and taking the max over $u\in\mathcal Z_h$ gives
\begin{equation}
\label{eq:cons_same_grid}
\Big|\max_{u\in\mathcal Z_h}\mathcal{I}[\mathcal{Q}_h]\big(x^{m,+}_{k,j}(u)\big)
-\max_{u\in\mathcal Z_h}\mathcal{Q}\big(x^{m,+}_{k,j}(u)\big)\Big|
\;\le\;C_q h^2+C_{\mathrm{int}}h^2+C_g h^{\,1+\kappa}.
\end{equation}
Finally, since $u\mapsto x^{m,+}_{k,j}(u)$ is Lipschitz on $[0,1]$ and $x\mapsto\mathcal Q(x)$ is Lipschitz on $\Omega$,
the control discretization satisfies
\begin{equation}
\label{eq:cons_control}
\Big|\sup_{u\in\mathcal Z}\mathcal{Q}\big(x^{m,+}_{k,j}(u)\big)
-\max_{u\in\mathcal Z_h}\mathcal{Q}\big(x^{m,+}_{k,j}(u)\big)\Big|
\;\le\; C_u\,\Delta u
=\mathcal O(h).
\end{equation}
Adding \eqref{eq:cons_same_grid} and \eqref{eq:cons_control} yields the stated
$\mathcal O(h+h^{1+\kappa}+h^2)$ bound.
\end{proof}


\small
\begin{thebibliography}{10}

\bibitem{AbramowitzStegun1972}
{\sc M.~Abramowitz and I.~A. Stegun}, {\em {Handbook of mathematical
  functions}}, {Dover, New York}, 1972.

\bibitem{barles-souganidis:1991}
{\sc G.~Barles and P.~Souganidis}, {\em Convergence of approximation schemes
  for fully nonlinear equations}, Asymptotic Analysis, 4 (1991), pp.~271--283.

\bibitem{barndorff1997normal}
{\sc O.~E. Barndorff-Nielsen}, {\em Normal inverse gaussian distributions and
  stochastic volatility modelling}, Scandinavian Journal of statistics, 24
  (1997), pp.~1--13.

\bibitem{BenTalTeboulle2007}
{\sc A.~Ben-Tal and M.~Teboulle}, {\em An old-new concept of convex risk
  measures: The optimized certainty equivalent}, Mathematical Finance, 17
  (2007), pp.~449--476.

\bibitem{carr2002fine}
{\sc P.~Carr, H.~Geman, D.~B. Madan, and M.~Yor}, {\em The fine structure of
  asset returns: An empirical investigation}, The Journal of Business, 75
  (2002), pp.~305--332.

\bibitem{chen08a}
{\sc Z.~Chen and P.~A. Forsyth}, {\em A numerical scheme for the impulse
  control formulation for pricing variable annuities with a {Guaranteed Minimum
  Withdrawal Benefit (GMWB)}}, Numerische Mathematik, 109 (2008), pp.~535--569.

\bibitem{dang2026multi}
{\sc D.-M. Dang and C.~Chen}, {\em Multi-period mean-buffered probability of
  exceedance in {Defined Contribution} portfolio optimization}, SIAM Journal on
  Financial Mathematics,  (2026).
\newblock to appear.

\bibitem{DangForsyth2014}
{\sc D.-M. Dang and P.~Forsyth}, {\em Continuous time mean-variance optimal
  portfolio allocation under jump diffusion: A numerical impulse control
  approach}, Numerical Methods for Partial Differential Equations, 30 (2014),
  pp.~664--698.

\bibitem{dang2016better}
{\sc D.-M. Dang and P.~A. Forsyth}, {\em Better than pre-commitment
  mean-variance portfolio allocation strategies: A semi-self-financing
  {Hamilton--Jacobi--Bellman} equation approach}, European Journal of
  Operational Research, 250 (2016), pp.~827--841.

\bibitem{dang2025monotone}
{\sc D.-M. Dang and H.~Zhou}, {\em A monotone piecewise constant control
  integration approach for the two-factor uncertain volatility model}, IMA
  Journal of Numerical Analysis,  (2025),
  \url{https://doi.org/10.1093/imanum/draf095}.

\bibitem{du2025fourier}
{\sc R.~Du and D.-M. Dang}, {\em Fourier neural network approximation of
  transition densities in finance}, SIAM Journal on Scientific Computing, 47
  (2025), pp.~C529--C557.

\bibitem{Fang2008}
{\sc F.~Fang and C.~Oosterlee}, {\em A novel pricing method for {E}uropean
  options based on {F}ourier-{C}osine series expansions}, SIAM Journal on
  Scientific Computing, 31 (2008), pp.~826--848.

\bibitem{filippi2020conditional}
{\sc C.~Filippi, G.~Guastaroba, and M.~G. Speranza}, {\em Conditional
  value-at-risk beyond finance: a survey}, International Transactions in
  Operational Research, 27 (2020), pp.~1277--1319.

\bibitem{forsyth2020multiperiod}
{\sc P.~A. Forsyth}, {\em Multiperiod mean {Conditional Value at Risk} asset
  allocation: Is it advantageous to be time consistent?}, SIAM Journal on
  Financial Mathematics, 11 (2020), pp.~358--384.

\bibitem{ken1999levy}
{\sc S.~Ken-Iti}, {\em L{\'e}vy processes and infinitely divisible
  distributions}, vol.~68, Cambridge University Press, 1999.

\bibitem{kinga2015method}
{\sc D.~Kinga, J.~B. Adam, et~al.}, {\em A method for stochastic optimization},
  in International conference on learning representations (ICLR), vol.~5,
  California;, 2015.

\bibitem{kleindorfer2005managing}
{\sc P.~R. Kleindorfer and G.~H. Saad}, {\em Managing disruption risks in
  supply chains}, Production and Operations Management, 14 (2005), pp.~53--68.

\bibitem{kou01}
{\sc S.~Kou}, {\em A jump diffusion model for option pricing}, {M}anagement
  {S}cience, 48 (2002), pp.~1086--1101.

\bibitem{labadie2004optimal}
{\sc J.~W. Labadie}, {\em Optimal operation of multireservoir systems:
  State-of-the-art review}, Journal of Water Resources Planning and Management,
  130 (2004), pp.~93--111.

\bibitem{lu2024semi}
{\sc Y.~Lu and D.-M. Dang}, {\em {A semi-Lagrangian $\varepsilon$-monotone
  Fourier method for continuous withdrawal GMWBs under jump-diffusion with
  stochastic interest rate}}, Numerical Methods for Partial Differential
  Equations, 40 (2024), p.~e23075.

\bibitem{MaForsyth2015}
{\sc K.~Ma and P.~Forsyth}, {\em An unconditionally monotone numerical scheme
  for the two-factor uncertain volatility model}, IMA Journal of Numerical
  Analysis, 37(2) (2017), pp.~905--944.

\bibitem{bpoe2018}
{\sc A.~Mafusalov and S.~Uryasev}, {\em Buffered probability of exceedance:
  Mathematical properties and optimization}, SIAM Journal on Optimization, 28
  (2018), pp.~1077--1103.

\bibitem{mclachlan2019finite}
{\sc G.~J. McLachlan, S.~X. Lee, and S.~I. Rathnayake}, {\em Finite mixture
  models}, Annual review of statistics and its application, 6 (2019),
  pp.~355--378.

\bibitem{merton1976option}
{\sc R.~C. Merton}, {\em Option pricing when underlying stock returns are
  discontinuous}, Journal of financial economics, 3 (1976), pp.~125--144.

\bibitem{miller2017optimal}
{\sc C.~Miller and I.~Yang}, {\em Optimal control of conditional value-at-risk
  in continuous time}, SIAM Journal on Control and Optimization, 55 (2017),
  pp.~856--884.

\bibitem{oberman2006convergent}
{\sc A.~M. Oberman}, {\em Convergent difference schemes for degenerate elliptic
  and parabolic equations: Hamilton--jacobi equations and free boundary
  problems}, SIAM Journal on Numerical Analysis, 44 (2006), pp.~879--895.

\bibitem{RT2000}
{\sc R.~Rockafellar and S.~Uryasev}, {\em Optimization of conditional
  value-at-risk}, Journal of Risk, 2 (2000), pp.~21--41.

\bibitem{rockafellar2000optimization}
{\sc R.~T. Rockafellar, S.~Uryasev, et~al.}, {\em Optimization of conditional
  value-at-risk}, Journal of risk, 2 (2000), pp.~21--42.

\bibitem{ruijter2012two}
{\sc M.~J. Ruijter and C.~W. Oosterlee}, {\em Two-dimensional {Fourier} cosine
  series expansion method for pricing financial options}, SIAM Journal on
  Scientific Computing, 34 (2012), pp.~B642--B671.

\bibitem{stein2011fourier}
{\sc E.~M. Stein and R.~Shakarchi}, {\em Fourier analysis: an introduction},
  vol.~1, Princeton University Press, 2011.

\bibitem{tran2019convergence}
{\sc P.~T. Tran et~al.}, {\em On the convergence proof of amsgrad and a new
  version}, IEEE Access, 7 (2019), pp.~61706--61716.

\bibitem{yosida1968functional}
{\sc K.~Yosida}, {\em Functional analysis}, vol.~123, Springer Science \&
  Business Media, 2012.

\bibitem{zhang2024monotone}
{\sc H.~Zhang and D.~Dang}, {\em A monotone numerical integration method for
  mean-variance portfolio optimization under jump-diffusion models},
  Mathematics and Computers in Simulation, 219 (2024), pp.~112--140.

\bibitem{zhou2025numerical}
{\sc H.~Zhou and D.-M. Dang}, {\em Numerical analysis of {American} option
  pricing in a two-asset jump-diffusion model}, Applied Numerical Mathematics,
  216 (2025), pp.~98--126.

\end{thebibliography}

\vfill
\end{document}